\documentclass[11pt]{article}
\usepackage{amsmath,amsthm}
\usepackage{appendix}
\usepackage{amssymb}
\usepackage{graphicx}
\usepackage{epsfig}
\usepackage{color}
\usepackage[normalem]{ulem}
\usepackage{subfig}
\usepackage{newfloat}
\usepackage{algorithm}
\usepackage{algorithmic}
\usepackage{caption}
\usepackage{booktabs}
\usepackage[detect-all]{siunitx}
\usepackage{etoolbox}
\usepackage{multirow}
\usepackage{mathtools}
\usepackage{tabularx}
\usepackage{array}
\usepackage{enumerate} 
\usepackage{comment} 
\usepackage{url}
\usepackage{hyperref}
\usepackage[a4paper,tmargin=1.5in,bmargin=1.5in]{geometry}

\newtheorem{theo}{Theorem}[section] 

\newtheorem{rema}[theo]{Remark}

\newcommand{\beq}{\begin{equation}}
\newcommand{\eeq}{\end{equation}}
\newcommand{\beqs}{\begin{equation*}}
\newcommand{\eeqs}{\end{equation*}}

\newcommand{\C}{\mathbb C}
\newcommand{\R}{\mathbb R}
\newcommand{\bigO}{{\cal O}}
\newcommand{\imagunit}{{\bf i}}

\newcommand{\hinfsym}{\mathcal{H}}
\newcommand{\linfsym}{\mathcal{L}}
\newcommand{\hinfcal}{{\hinfsym_\infty}}
\newcommand{\Hinf}{{\hinfsym_\infty}}
\newcommand{\Linf}{{\linfsym_\infty}}

\newcommand{\Cmn}[2]{\C^{{#1}\times{#2}}}

\renewcommand\Re{\operatorname{Re}}
\renewcommand\Im{\operatorname{Im}}

\newcommand{\Real}[1]{\Re{\left({#1}\right)}}

\DeclareMathOperator*{\argmax}{arg\,max}

\DeclareFloatingEnvironment[placement=!th]{algfloat}
\newcommand{\algnote}[1]{\footnotesize \sc{Note: \it#1 } }

\def\compleib{$COMPL_eib$}
\def\hifoo{{\sc hifoo}}
\def\hifood{{\sc hifoo}d}

\def\ntfc{g_\mathrm{c}}
\def\ntfd{g_\mathrm{d}}
\def\tfc{G_\mathrm{c}}
\def\tfd{G_\mathrm{d}}
\def\eigderivmat{H}
\def\nopt{\phi}

\def\Mc{\mathcal{M}_\gamma}
\def\Nc{\mathcal{N}}
\def\Md{\mathcal{S}_\gamma}
\def\Nd{\mathcal{T}_\gamma}

\def\MNpencont{(\Mc,\Nc)}
\def\MNpendisc{(\Md,\Nd)}
\def\MNcontnp{\Mc,\Nc}

\newcommand{\tfs}[1]{C\left(#1 E -A\right)^{-1}B + D}
\def\eitheta{e^{\imagunit \theta}}
\def\eithetaconj{e^{-\imagunit \theta}}

\newcommand{\mytilde}{\raise.17ex\hbox{$\scriptstyle\mathtt{\sim}$}}
\newcommand{\matlab}{{MATLAB}}

\title{Faster and more accurate computation of the $\Hinf$ norm via optimization}
\author{
Peter Benner\thanks{
Max Planck Institute for Dynamics of Complex Technical Systems, Magdeburg, 39106 Germany (benner@mpi-magdeburg.mpg.de).}
\and 
Tim Mitchell\thanks{
Max Planck Institute for Dynamics of Complex Technical Systems, Magdeburg, 39106 Germany (mitchell@mpi-magdeburg.mpg.de).}
\date{July 18, 2018}
}
\begin{document} 
\maketitle
\begin{abstract}
In this paper, we propose an improved method for computing the $\Hinf$ norm of 
linear dynamical systems that results in a code that is often several times faster 
than existing methods.  
By using standard optimization tools to rebalance 
the work load of the standard algorithm due to Boyd, Bala\-krish\-nan, Bru\-insma, and 
Stein\-buch, we aim to minimize
the number of expensive eigenvalue computations that must be performed.  
Unlike the standard algorithm, our modified approach can also 
calculate the $\Hinf$ norm to full precision with little extra work, and also offers more opportunity 
to further accelerate its performance via parallelization. 
Finally, 
we demonstrate that the local optimization we have employed 
to speed up the standard globally-convergent algorithm
can also be an effective strategy on its own
for approximating the $\Hinf$ norm of large-scale systems.
\end{abstract}

\section{Introduction}
\label{sec:intro}

Consider the continuous-time linear dynamical system 
\begin{subequations}
\label{eq:lti_cont}
\begin{align}
E\dot{x} &=  Ax + Bu \\
y &=  Cx + Du,
\end{align}
\end{subequations}
where $A \in \Cmn{n}{n}$, $B \in \Cmn{n}{m}$, $C \in \Cmn{p}{n}$, $D \in \Cmn{p}{m}$, and $E \in \Cmn{n}{n}$.
The system defined by \eqref{eq:lti_cont} arises in many engineering applications 
and as a consequence, there has been a strong motivation for fast methods to compute properties 
that measure the sensitivity of the system or its robustness to noise.  Specifically, 
a quantity of great interest is the $\Hinf$ norm, which is defined as 
\beq
	\label{eq:hinf_cont}
	\|G\|_{\hinfcal} \coloneqq \sup_{\omega \in \R} \| G(\imagunit \omega)\|_2,
\eeq
where
\beq
	G(\lambda) = C(\lambda E - A)^{-1} B + D 
\eeq
is the associated \emph{transfer function} of the system given by \eqref{eq:lti_cont}.  
The $\Hinf$ norm measures the maximum sensitivity of the system; 
in other words, the higher the value of the $\Hinf$ norm, the less robust the system is,
an intuitive interpretation considering that the $\Hinf$ norm is in fact the reciprocal
of the \emph{complex stability radius}, which itself is a generalization of the \emph{distance to instability} \cite[Section 5.3]{HinP05}.
The $\Hinf$ norm is also a key metric for assessing the quality of reduced-order
models that attempt to capture/mimic the dynamical behavior of large-scale systems, 
see, e.g., \cite{morAnt05,morBenCOW17}.
Before continuing, 
as various matrix pencils of the form $\lambda B - A$ will feature frequently in this work,
we use notation $(A,B)$ to abbreviate them.

When $E = I$, the $\Hinf$ norm is finite as long as $A$ is stable, whereas an 
unstable system would be considered infinitely sensitive.  If $E\ne I$, then 
\eqref{eq:lti_cont} is called a \emph{descriptor system}.
Assuming that $E$ is singular but $(A,E)$ is regular and at most index 1, then \eqref{eq:hinf_cont} still 
yields a finite value, provided that all the controllable and observable eigenvalues of $(A,E)$ 
are finite and in the open left half plane, where for $\lambda$ an eigenvalue of $(A,E)$ 
with right and left eigenvectors $x$ and $y$, $\lambda$ is considered 
\emph{uncontrollable} if $B^*y = 0$ and \emph{unobservable}
if $Cx=0$.  However, the focus of this paper is not about detecting when \eqref{eq:hinf_cont}
is infinite or finite, but to introduce an improved method for computing the $\Hinf$ norm when it is finite.
Thus, for conciseness in presenting our improved method, we will assume in this paper that
any system provided to an algorithm has a finite $\Hinf$ norm, as checking whether it is infinite can be considered a preprocessing step.\footnote{We note that while our proposed improvements are also directly applicable for computing the 
$\Linf$ norm, we will restrict the discussion here to just the $\Hinf$ norm for brevity.}

While the first algorithms \cite{BoyBK89, BoyB90,BruS90} for computing the $\Hinf$ norm date back to nearly 30 years ago,
there has been continued interest in improved methods, particularly as the state-of-art methods remain quite expensive 
with respective to their dimension $n$, meaning that computing the $\Hinf$ norm is generally only possible for rather small-dimensional systems.  In 1998, \cite{GenVV98} proposed an interpolation refinement to the existing algorithm of \cite{BoyB90,BruS90} to accelerate its rate of convergence.  
In the following year, for the special case of the distance to instability where $B=C=E=I$ and $D=0$, 
\cite{HeW99} used an inverse iteration to successively obtain increasingly better locally optimal approximations as 
a way of reducing the number of expensive Hamiltonian eigenvalue decompositions;
as we will discuss in Section~\ref{sec:hybrid_opt}, this method shares some similarity with the approach we propose here.
More recently, in 2011, \cite{BelP11} presented an entirely
different approach to computing the $\Hinf$ norm, by finding isolated common zeros of two certain bivariate polynomials.  
While they showed that their method was much faster than an implementation of \cite{BoyB90,BruS90}
on two SISO examples (single-input, single-output, that is, $m=p=1$), more comprehensive benchmarking does not appear to have been done yet.  Shortly thereafter, \cite{BenSV12} extended the now standard algorithm of 
\cite{BoyB90,BruS90} to descriptor systems.  
There also has been a very recent surge of interest in efficient $\Hinf$ norm approximation methods for large-scale systems.  
These methods fall into two 
broad categories: those that are applicable for descriptor systems with possibly singular $E$ matrices but require solving linear systems \cite{BenV14,FreSV14,AliBMetal17} and those that don't solve linear systems but require that $E=I$ or that $E$ is at least cheaply inverted \cite{GugGO13,MitO16}.  
Our contribution in this paper is twofold.
First, we improve upon the exact algorithms of \cite{BoyB90,BruS90,GenVV98} to not only 
compute the $\Hinf$ norm significantly faster but also obtain its value to machine precision with negligible 
additional expense (a notable difference compared to these earlier methods).
This is accomplished by incorporating local optimization techniques within these algorithms,
a change that also makes our new approach more amenable to additional acceleration via parallelization.
Second, that standard local optimization 
can even be used on its own 
to efficiently obtain locally optimal approximations to the 
$\Hinf$ norm of large-scale systems,
a simple and direct approach that has surprisingly not yet been considered 
and is even embarrassingly parallelizable.

The paper is organized as follows. In Sections~\ref{sec:bbbs} and \ref{sec:costs}, we describe 
the standard algorithms for computing the $\Hinf$ norm and then give an overview of their computational costs.
In Section~\ref{sec:hybrid_opt}, we introduce our new approach to computing the $\Hinf$ norm via leveraging
local optimization techniques.  Section~\ref{sec:discrete} describes how the results and algorithms are adapted for
discrete-time problems.  We present numerical results in Section~\ref{sec:numerical} 
for both continuous- and discrete-time problems.
Section~\ref{sec:hinf_approx} provides the additional experiments 
demonstrating how 
local optimization can also be a viable strategy for approximating the $\Hinf$ norm of 
large-scale systems.
Finally, in Section~\ref{sec:parallel} we discuss how significant speedups 
can be obtained for some problems when using parallel processing with our new approach, in contrast to 
the standard algorithms, which benefit very little from multiple cores.
Concluding remarks are given in Section~\ref{sec:wrapup}.

\section{The standard algorithm for computing the $\Hinf$ norm} 
\label{sec:bbbs}
We begin by presenting a key theorem relating the singular values of the transfer function 
to purely imaginary eigenvalues of an associated matrix pencil.  For the case of simple ODEs, where
$B=C=E=I$ and $D=0$, the result goes back to \cite[Theorem 1]{Bye88}, and was first extended
to linear dynamical systems with input and output with $E=I$ in \cite[Theorem 1]{BoyBK89}, and 
then most recently generalized to systems where $E\ne I$ in \cite[Theorem 1]{BenSV12}.
We state the theorem without the proof, since it is readily available in \cite{BenSV12}.

\begin{theo}
\label{thm:eigsing_cont}
Let $\lambda E - A$ be regular with no finite eigenvalues on the imaginary axis, $\gamma > 0$ not a singular value of $D$, and $\omega \in \R$.
Consider the matrix pencil $\MNpencont$, where
\beq
	\label{eq:MNpencil_cont}
	\Mc \coloneqq \begin{bmatrix} 	A - BR^{-1}D^*C 	& -\gamma BR^{-1}B^* \\ 
						\gamma C^*S^{-1}C 	& -(A - BR^{-1}D^*C)^* \end{bmatrix} 
	~\text{and}~ 
	\Nc \coloneqq  \begin{bmatrix} E & 0\\ 0 & E^*\end{bmatrix} 
	\eeq
and $R = D^*D - \gamma^2 I$ and $S = DD^* - \gamma^2 I$.
Then $\imagunit \omega$ is an eigenvalue of matrix pencil $\MNpencont$ if and only if
$\gamma$ is a singular value of $G(\imagunit \omega)$. 
\end{theo}

Theorem~\ref{thm:eigsing_cont} immediately leads to an algorithm for computing the $\Hinf$ norm
based on computing the imaginary eigenvalues, if any, of the associated matrix pencil \eqref{eq:MNpencil_cont}.
For brevity in this section, we assume that $\max \|G(\imagunit \omega)\|_2$ is not attained at $\omega = \infty$, in which case
the $\Hinf$ norm would be $\|D\|_2$.  
 Evaluating the norm of the transfer function for any finite frequency along the imaginary axis immediately gives a lower bound to the $\Hinf$ norm 
while an upper bound can be obtained 
by successively increasing $\gamma$ until the matrix pencil given by \eqref{eq:MNpencil_cont} no longer has any purely imaginary eigenvalues.  Then, it is straightforward to compute the $\Hinf$ norm using bisection, as first proposed in \cite{BoyBK89} 
and was inspired by the breakthrough result of \cite{Bye88} for computing the 
\emph{distance to instability} (i.e. the reciprocal of the $\Hinf$ norm for the special case of $B=C=E=I$ and $D=0$).

As Theorem~\ref{thm:eigsing_cont} provides a way to calculate all the frequencies where $\|G(\imagunit \omega)\|_2 = \gamma$, it 
was shortly thereafter proposed in \cite{BoyB90, BruS90} that instead of computing an upper bound and then using bisection, 
the initial lower bound could be successively increased in a monotonic fashion to the value of the 
$\Hinf$ norm.  
For convenience, it will be helpful to establish the following notation for the transfer function and its 
largest singular value, both as parameters of frequency $\omega$:
\begin{align}
	\label{eq:tf_cont}
	\tfc(\omega) {}& \coloneqq G(\imagunit \omega) \\
	\label{eq:ntf_cont}
	\ntfc(\omega) {}& \coloneqq \| G(\imagunit \omega) \|_2 = \| \tfc(\omega) \|_2.
\end{align}
Let $\{ \omega_1,\ldots,\omega_l\}$ be the set of imaginary parts of the purely 
imaginary eigenvalues of \eqref{eq:MNpencil_cont} for the initial value $\gamma$, sorted in increasing order.  
Considering the 
intervals $I_k = [\omega_k,\omega_{k+1}]$, \cite{BruS90} proposed increasing $\gamma$ via:
\beq
	\label{eq:gamma_mp}
 	\gamma_\mathrm{mp} = \max \ntfc(\hat\omega_k) 
 	\qquad \text{where} \qquad 
	\hat\omega_k ~\text{are the midpoints of the intervals}~ I_k.
\eeq
Simultaneously and independently, a similar
algorithm was proposed by \cite{BoyB90}, with the additional results that 
(a) it was possible to calculate which intervals $I_k$ satisfied $\ntfc(\omega) \ge \gamma$ 
for all $\omega \in I_k$, thus reducing the number of evaluations of $\ntfc(\omega)$ needed
at every iteration, and (b) this midpoint scheme actually had a local quadratic rate of convergence, 
greatly improving upon the linear rate of convergence of the earlier, bisection-based method.
This midpoint-based method, which we refer to as the BBBS algorithm for its authors Boyd, Bala\-krish\-nan, 
Bru\-insma, and Stein\-buch, is now considered the standard algorithm for computing the $\Hinf$ norm and
it is the algorithm implemented in the \matlab\ Robust Control Toolbox, e.g. routine \texttt{hinfnorm}.
Algorithm~\ref{alg:bbbs} provides a high-level pseudocode description for the standard BBBS algorithm while Figure~\ref{fig:bbbs_std} provides a corresponding pictorial description of how the method works.

\begin{algfloat}
\begin{algorithm}[H]
\floatname{algorithm}{Algorithm}
\caption{The Standard BBBS Algorithm}
\label{alg:bbbs}
\begin{algorithmic}[1]
	\REQUIRE{  
		$A \in \Cmn{n}{n}$, $B \in \Cmn{n}{m}$, $C \in \Cmn{p}{n}$, $D \in \Cmn{p}{m}$, 
		$E \in \Cmn{n}{n}$ and $\omega_0 \in \R$.
	}
	\ENSURE{ 
		$\gamma = \| G \|_\hinfcal$ and $\omega$ such that 
		$\gamma = \ntfc(\omega)$.  
		\\ \quad
	}
	
	\STATE $\gamma = \ntfc(\omega_0)$
	\WHILE {not converged} 
		\STATE \COMMENT{Compute the intervals that lie under $\ntfc(\omega)$ using 
						eigenvalues of the pencil:}	
		\STATE Compute $\Lambda_\mathrm{I} =
				\{ \Im \lambda : \lambda \in \Lambda(\MNcontnp) ~\text{and}~ \Re \lambda = 0\}$.
		\STATE Index and sort $\Lambda_\mathrm{I} = \{\omega_1,...,\omega_l\}$ s.t. 
				$\omega_j \le \omega_{j+1}$.
		\STATE Form all intervals $I_k = [\omega_k, \omega_{k+1}]$ s.t.
			each interval at height $\gamma$ is below $\ntfc(\omega)$.
			\label{algline:bbbs_ints}		
		\STATE \COMMENT{Compute candidate frequencies of the level-set intervals $I_k$:}
		\STATE Compute midpoints $\hat\omega_k = 0.5(\omega_k +\omega_{k+1})$ for each interval $I_k$.
			\label{algline:bbbs_points}
		\STATE \COMMENT{Update to the highest gain evaluated at these candidate frequencies:}
		\STATE $\omega = \argmax \ntfc(\hat\omega_k)$.
		\STATE $\gamma = \ntfc(\omega)$.
	\ENDWHILE
\end{algorithmic}
\end{algorithm}
\algnote{
The quartically converging variant proposed by \cite{GenVV98} replaces the midpoints of $I_k$ 
with the maximizing frequencies of Hermite cubic interpolants, which are uniquely 
determined by interpolating the values of $\ntfc(\omega)$ and $\ntfc^\prime(\omega)$ 
at both endpoints of each interval $I_k$.
}
\end{algfloat}

\begin{figure}[!t]
\subfloat[The standard BBBS method]{
\includegraphics[scale=.37,trim={1cm 0 2cm 0},clip]{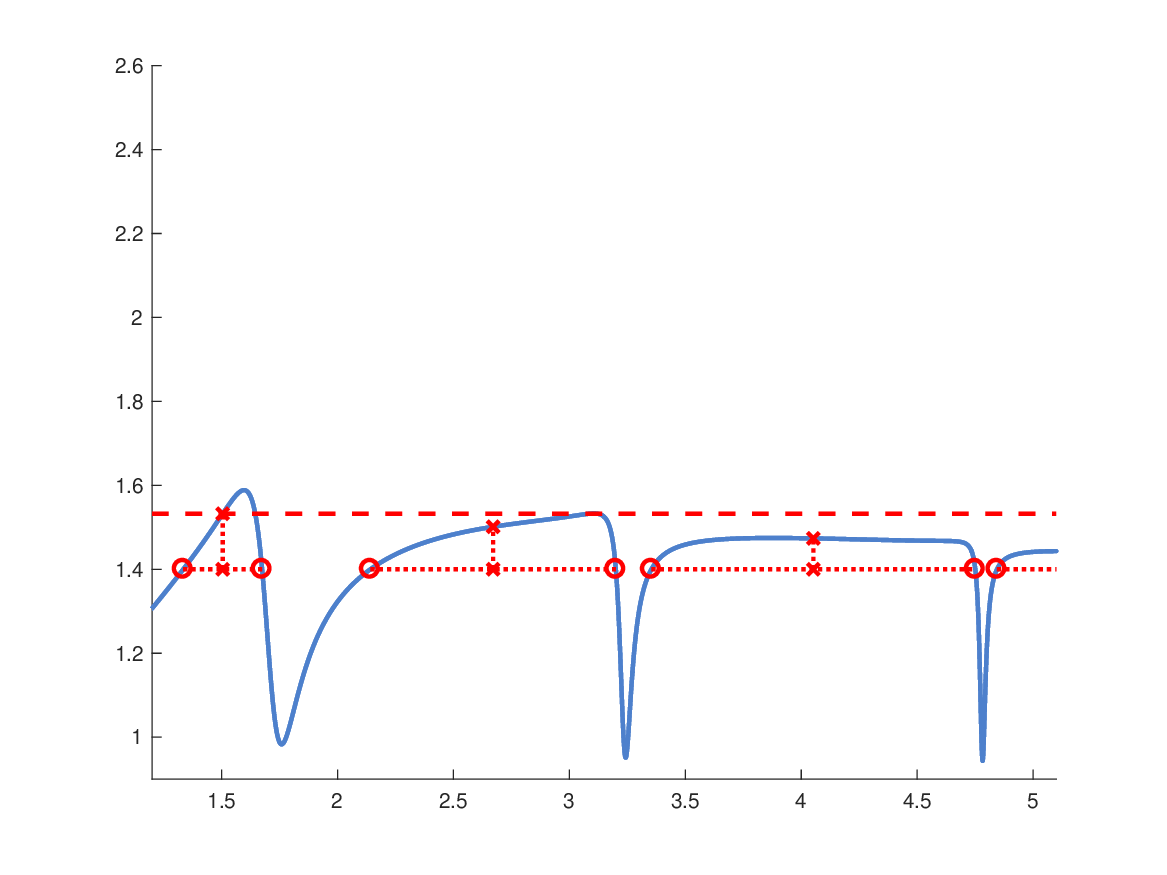} 
\label{fig:bbbs_std}
}
\subfloat[With cubic interpolation]{
\includegraphics[scale=.37,trim={1cm 0 2cm 0},clip]{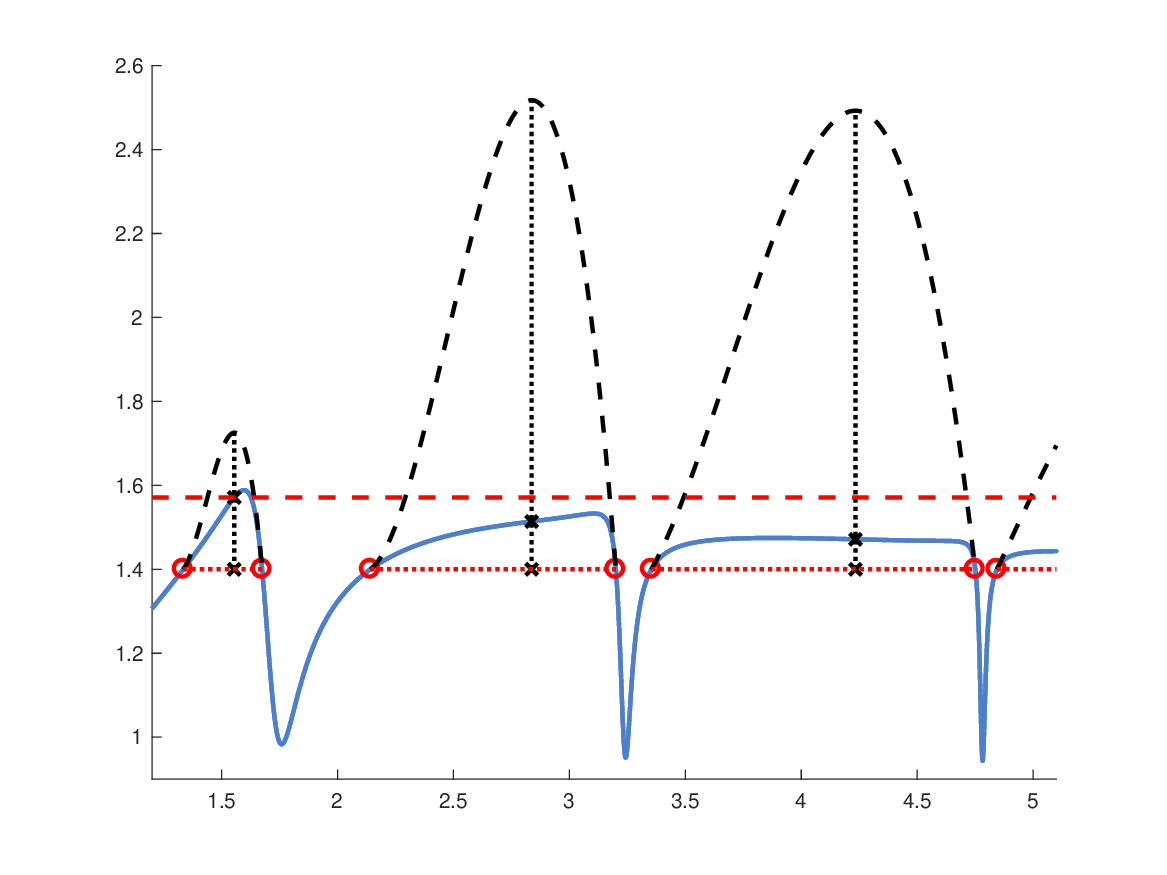}
\label{fig:bbbs_interp}
}
\caption{The blue curves show the value of $\ntfc(\omega) = \|G(\imagunit \omega)\|_2$ while the red circles
mark the frequencies $\{\omega_1,\ldots,\omega_l\}$ where $\ntfc(\omega) = \gamma = 1.4$, 
computed by taking the imaginary parts of the purely
imaginary eigenvalues of \eqref{eq:MNpencil_cont} on a sample problem.  
The left plot shows the midpoint scheme of the standard BBBS algorithm to obtain an increased value 
for $\gamma$, depicting by the dashed horizontal line.  The right plot shows the cubic interpolation
refinement of \cite{GenVV98} for the same problem and initial value of $\gamma$.  
The dashed black curves depict the cubic Hermite interpolants for each interval while the dotted vertical lines 
show their respective maximizing frequencies.  As can be seen, the cubic interpolation scheme results 
in a larger increase in $\gamma$ (again represented by the dashed horizontal line) compared to the standard BBBS method
on this iterate for this problem.
}
\label{fig:bbbs_plots}
\end{figure}

In \cite{GenVV98}, a refinement to the BBBS algorithm was proposed which increased its local quadratic rate of convergence to quartic.  This was done by evaluating $\ntfc(\omega)$ at the maximizing frequencies
of the unique cubic Hermite interpolants for each level-set interval, instead of at the midpoints.  
That is, for each interval $I_k = [\omega_k,\omega_{k+1}]$, 
the unique interpolant $c_k(\omega) = c_3\omega^3 + c_2\omega^2 + c_1\omega + c_0$ is constructed so that
\beq
\label{eq:interp_cont}
\begin{aligned}
	c_k(\omega_k) &= \ntfc(\omega_k) \\ 
	c_k(\omega_{k+1}) &= \ntfc(\omega_{k+1})
\end{aligned}
\qquad \text{and} \qquad
\begin{aligned}
	c^\prime_k(\omega_k) &= \ntfc^\prime(\omega_k) \\ 
	c^\prime_k(\omega_{k+1}) &= \ntfc^\prime(\omega_{k+1}).
\end{aligned}
\eeq
Then, $\gamma$ is updated via 
\beq
	\label{eq:gamma_cubic}
	\gamma_\mathrm{cubic} = \max \ntfc(\hat\omega_k) 
	\qquad \text{where} \qquad 
	\hat\omega_k = \argmax_{\omega \in I_k} c_k(\omega),
\eeq
that is, $\hat\omega_k$ is now the maximizing value of interpolant $c_k(\omega)$ on its interval $I_k$,
which of course can be cheaply and explicitly computed.
In \cite{GenVV98}, the single numerical example shows 
the concrete benefit of this interpolation scheme, where the standard BBBS algorithm required
six eigenvalue decompositions of \eqref{eq:MNpencil_cont} to converge, while their new method only required four.
As only the selection of the $\omega_k$ values is different, the pseudocode for this improved version of 
the BBBS algorithm remains largely the same, as mentioned in the note of Algorithm~\ref{alg:bbbs}.
Figure~\ref{fig:bbbs_interp} provides a corresponding pictorial description of the cubic-interpolant-based refinement.

As it turns out, computing the derivatives in \eqref{eq:interp_cont} for the Hermite interpolations of each interval
can also be done with little extra work.
Let $u(\omega)$ and $v(\omega)$ be the associated left and right singular vectors 
corresponding to $\ntfc(\omega)$,
recalling that $\ntfc(\omega)$ is the largest singular value of $\tfc(\omega)$,
and assume that $\ntfc(\hat\omega)$ is a simple singular value 
for some value $\hat\omega \in \R$.
By standard perturbation theory 
(exploiting the equivalence of singular values of a matrix $A$ and eigenvalues of 
$\left[ \begin{smallmatrix} 0 & A \\ A^* & 0\end{smallmatrix} \right]$
and applying \cite[Theorem 5]{Lan64}), it then follows that
\beq
	\label{eq:ntfcprime_cont}
	\ntfc^\prime(\omega) \Big\rvert_{\omega=\hat\omega} = \Real{ u(\hat\omega)^* \tfc^\prime(\hat\omega) v(\hat\omega) },
\eeq
where, by standard matrix differentiation rules with respect to parameter $\omega$, 
\beq
	\tfc^\prime(\omega) = - \imagunit C \left(\imagunit \omega E - A \right)^{-1} E \left(\imagunit \omega E - A \right)^{-1} B.
\eeq
As shown in \cite{SreVT95}, it is actually fairly cheap to compute \eqref{eq:ntfcprime_cont} 
if the eigenvectors corresponding to the purely imaginary eigenvalues of \eqref{eq:MNpencil_cont} have also been computed, as there is a correspondence between 
these eigenvectors
and the associated 
singular vectors for $\gamma$.  For $\gamma = \ntfc(\hat\omega)$,
if $\left[ \begin{smallmatrix} q \\ s \end{smallmatrix} \right]$ is an eigenvector of \eqref{eq:MNpencil_cont} for imaginary 
eigenvalue $\imagunit \hat\omega$, then the equivalences
\beq
	\label{eq:qs_vecs}
	q = \left( \imagunit \hat\omega E - A \right)^{-1}Bv(\hat\omega)
	\quad \text{and} \quad 
	s = \left( \imagunit \hat\omega E - A \right)^{-*}C^*u(\hat\omega)
\eeq
both hold, where $u(\hat\omega)$ and $v(\hat\omega)$ are left and right singular vectors 
associated with singular value $\ntfc(\hat\omega)$.  (To see why these equivalences hold, we refer the reader
to the proof of Theorem~\ref{thm:eigsing_disc} for the discrete-time analog result.)
Thus, \eqref{eq:ntfcprime_cont} may be rewritten as follows:
\begin{align}
	\label{eq:ntfcprime_defn_cont}
	\ntfc^\prime(\omega)\Big\rvert_{\omega=\hat\omega} &= 
		\Real{ u(\hat\omega)^* \tfc^\prime(\hat\omega) v(\hat\omega)} \\
	\label{eq:ntfcprime_direct_cont}
	 &= -\Real{ u(\hat\omega)^* \imagunit C \left(\imagunit \hat\omega E - A \right)^{-1} E \left(\imagunit \hat \omega E - A \right)^{-1} B v(\hat\omega)} \\
	 \label{eq:ntfcprime_eig_cont}
	&= -\Real{\imagunit s^* E q},
\end{align}
and it is thus clear that \eqref{eq:ntfcprime_cont} is cheaply computable for all the endpoints of the intervals $I_k$, 
provided that the eigenvalue decomposition of \eqref{eq:MNpencil_cont} has already been computed.  
 
\section{The computational costs involved in the BBBS algorithm}
\label{sec:costs}
The main drawback of the BBBS algorithm is its algorithmic complexity, which is $\bigO(n^3)$ work
per iteration.  This not only limits the tractability of computing the $\Hinf$ norm to rather low-dimensional (in $n$) 
systems but can also make computing the $\Hinf$ norm to full precision an expensive proposition for 
even moderately-sized systems.  In fact, the default tolerance for \texttt{hinfnorm} in \matlab\
is set quite large, 0.01, presumably to keep its runtime as fast as possible, at the expense of 
sacrificing accuracy.  In Table~\ref{table:hinfnorm_tol}, we 
report the relative error of computing the $\Hinf$ norm when using \texttt{hinfnorm}'s default
tolerance of $0.01$ compared to $10^{-14}$, along with the respective runtimes for several test problems,
observing that computing the $\Hinf$ norm to near full precision can often take between two to three times longer.
While computing only a handful of the most significant digits of the $\Hinf$ norm 
may be sufficient for some applications, this is certainly not true in general.  Indeed, the source
code for HIFOO \cite{BurHLetal06}, which designs $\Hinf$ norm fixed-order optimizing controllers for a given 
open-loop system via nonsmooth optimization, specifically contains the comment regarding
\texttt{hinfnorm}: ``default is .01, which is too crude".  In HIFOO, the $\Hinf$ norm is minimized
by updating the controller variables at every iteration but the optimization method assumes that the objective
function is continuous; if the $\Hinf$ norm is not calculated sufficiently accurately, then 
it may appear to be discontinuous, which can cause the underlying optimization method to break down.  
Thus there is motivation to not only improve the overall runtime of computing the $\Hinf$ norm for large 
tolerances, but also to make the computation as fast as possible when computing the $\Hinf$ norm to full precision.

\begin{table}
\centering
\begin{tabular}{ l | rrr | c | c | SS } 
\toprule
\multicolumn{8}{c}{\texttt{$h$ = hinfnorm($\cdot$,1e-14)} versus \texttt{$\hat h$ = hinfnorm($\cdot$,0.01)}}\\
\midrule
\multicolumn{1}{c}{} & \multicolumn{3}{c}{Dimensions} & \multicolumn{1}{c}{} & 
\multicolumn{1}{c}{Relative Error} & \multicolumn{2}{c}{Wall-clock time (sec.)}\\
\cmidrule(lr){2-4}
\cmidrule(lr){6-6}
\cmidrule(lr){7-8}
\multicolumn{1}{l}{Problem} & 
	\multicolumn{1}{c}{$n$} & 
	\multicolumn{1}{c}{$m$} &
	\multicolumn{1}{c}{$p$} & 
	\multicolumn{1}{c}{$E=I$} &
	\multicolumn{1}{c}{$ \frac{\hat h - h}{h}$} & 
	\multicolumn{1}{c}{\texttt{tol=1e-14}} &
	\multicolumn{1}{c}{\texttt{tol=0.01}} \\
\midrule
\texttt{CSE2} & 63 & 1 & 32 & Y      & $-2.47 \times 10^{-4}$   &    0.137 &    0.022 \\ 
\texttt{CM3} & 123 & 1 & 3 & Y       & $-2.75 \times 10^{-3}$   &    0.148 &    0.049 \\ 
\texttt{CM4} & 243 & 1 & 3 & Y       & $-4.70 \times 10^{-3}$   &    1.645 &    0.695 \\ 
\texttt{ISS} & 270 & 3 & 3 & Y       & $-1.04 \times 10^{-6}$   &    0.765 &    0.391 \\ 
\texttt{CBM} & 351 & 1 & 2 & Y       & $-4.20 \times 10^{-5}$   &    3.165 &    1.532 \\ 
\texttt{randn 1} & 500 & 300 & 300 & Y & 0                        &   21.084 &   30.049 \\ 
\texttt{randn 2} & 600 & 150 & 150 & N & $-6.10 \times 10^{-8}$   &   31.728 &   16.199 \\ 
\texttt{FOM} & 1006 & 1 & 1 & Y      & $-1.83 \times 10^{-5}$   &  128.397 &   36.529 \\ 
\midrule
\texttt{LAHd} & 58 & 1 & 3 & Y       & $-7.36 \times 10^{-3}$   &    0.031 &    0.015 \\ 
\texttt{BDT2d} & 92 & 2 & 4 & Y      & $-7.67 \times 10^{-4}$   &    0.070 &    0.031 \\ 
\texttt{EB6d} & 170 & 2 & 2 & Y      & $-5.47 \times 10^{-7}$   &    0.192 &    0.122 \\ 
\texttt{ISS1d} & 280 & 1 & 273 & Y   & $-1.53 \times 10^{-3}$   &   16.495 &    3.930 \\ 
\texttt{CBMd} & 358 & 1 & 2 & Y      & $-2.45 \times 10^{-6}$   &    1.411 &    0.773 \\ 
\texttt{CM5d} & 490 & 1 & 3 & Y      & $-7.38 \times 10^{-3}$   &   10.802 &    2.966 \\
\bottomrule
\end{tabular}
\caption{For various problems, the relative error of  
computing the $\Hinf$ norm using \texttt{hinfnorm} with its quite loose default tolerance 
is shown. 
The first eight are continuous time problems while the last six, ending in \texttt{d}, are discrete time.
As can be seen, using \texttt{hinfnorm} to compute the $\Hinf$ norm to near machine accuracy
can often take between two to four times longer, and that this penalty is not necessarily related to 
dimension: for example, 
the running times are increased by factors of 3.02 and 3.51 for \texttt{CM3} and \texttt{FOM}, respectively,
despite that \texttt{FOM} is nearly ten times larger in dimension.
}
\label{table:hinfnorm_tol}
\end{table}

The dominant cost of the BBBS algorithm 
is computing the eigenvalues of \eqref{eq:MNpencil_cont} at
every iteration.  Even though the method converges quadratically, and quartically when using 
the cubic interpolation refinement, the eigenvalues of $\MNpencont$ will still 
generally be computed for multiple values of $\gamma$ before convergence, for either variant
of the algorithm.
Furthermore, pencil $\MNpencont$ is $2n \times 2n$, meaning that the 
$\bigO(n^3)$ work per iteration also contains a significantly larger constant factor;
computing the eigenvalues of a $2n \times 2n$ problem typically takes at least eight times
longer than a $n \times n$ one.
If cubic interpolation is used, 
computing the derivatives \eqref{eq:ntfcprime_cont}
via the eigenvectors of $\MNpencont$, as proposed by \cite{GenVV98}
using the equivalences in \eqref{eq:qs_vecs} and \eqref{eq:ntfcprime_eig_cont},
can sometimes be quite expensive as well.
If on a particular iteration, 
the number of purely imaginary eigenvalues of $\MNpencont$ 
is close to $n$, say $\hat n$, then assuming 64-bit computation, 
an additional $4 \hat{n}^2$ doubles of memory would 
be required to store these eigenvectors.\footnote{
Although computing eigenvectors with \texttt{eig} in \matlab\ is currently an all or none affair, 
LAPACK does provide the user the option to only
compute certain eigenvectors, so that all $2n$ eigenvectors would not always need to be computed.}
Finally, computing the purely imaginary eigenvalues of $\MNpencont$
using the regular QZ algorithm can be ill advised; in practice, rounding error in the real parts
of the eigenvalues can make it difficult to detect which of the computed eigenvalues are 
supposed to be the purely imaginary ones and which are merely just close to the imaginary axis.
Indeed, purely imaginary eigenvalues can easily be perturbed off of the imaginary axis when 
using standard QZ; \cite[Figure 4]{BenSV16} illustrates this issue particularly well.  Failure
to properly identify the purely imaginary eigenvalues can cause the BBBS algorithm to return
incorrect results. As such, it is instead recommended \cite[Section II.D]{BenSV12} to use the 
specialized Hamiltonian-structure-preserving eigensolvers of \cite{BenBMetal02,BenSV16} to avoid this
problem.  However, doing so can be even more expensive as it requires computing the eigenvalues 
of a related matrix pencil that is even larger: $(2n+m+p)\times(2n+m+p)$.  

On the other hand, computing \eqref{eq:ntf_cont}, the norm of the transfer function, is typically rather 
inexpensive, at least relative to computing the imaginary eigenvalues of the matrix pencil 
\eqref{eq:MNpencil_cont}; 
Table~\ref{table:ntf_vs_pencil} presents for data on how much faster computing the 
singular value decomposition of $G(\imagunit \omega)$  
can be
compared to computing 
the eigenvalues of $\MNpencont$ (using regular QZ), 
using randomly-generated systems composed of
dense matrices of various dimensions.  
In the first row of Table~\ref{table:ntf_vs_pencil}, we see that computing the eigenvalues of \eqref{eq:MNpencil_cont}
for tiny systems ($n=m=p=20$) can take up to two-and-a-half times longer than computing the SVD of $G(\imagunit \omega)$ on 
modern hardware 
and this disparity quickly grows larger as the dimensions are all increased 
(up to 36.8 faster for $n=m=p=400$).  Furthermore, for moderately-sized systems
where $m,p \ll n$ (the typical case in practice), 
the performance gap dramatically widens to up to 119 times faster
to compute the SVD of $G(\imagunit \omega)$ versus the eigenvalues of $\MNpencont$
(the last row of Table~\ref{table:ntf_vs_pencil}).
Of course, this disparity in runtime speeds is not surprising.  Computing the eigenvalues of \eqref{eq:MNpencil_cont} 
involves working with a $2n \times 2n$ (or larger when using structure-preserving eigensolvers) 
matrix pencil while the main costs to evaluate the norm of the transfer 
function at a particular frequency
involve first solving a linear system of dimension $n$ 
to compute either the $(\imagunit \omega E - A)^{-1}B$ or $C(\imagunit \omega E - A)^{-1}$ 
term in $G(\imagunit \omega)$ and then computing 
the maximum singular value of $G(\imagunit \omega)$, which is $p \times m$.  If $\max(m,p)$ is small, 
the cost to compute the largest singular value is negligible and even if $\max(m,p)$ is not small, the largest singular value 
can still  typically be  computed easily and efficiently using sparse methods.  Solving the $n$-dimensional linear system is 
typically going to be much cheaper than computing the eigenvalues of the $2n \times 2n$ pencil, and more so if 
$A$ and $E$ are not dense and $(\imagunit \omega E - A)$ permits a fast (sparse) LU decomposition.

\begin{table}
\setlength{\tabcolsep}{8pt}
\centering
\begin{tabular}{ccc|cc}
\toprule
\multicolumn{5}{c}{Computing $\|G(\imagunit \omega)\|_2$ versus \texttt{eig}($\MNcontnp$)}\\
\midrule
\multicolumn{3}{c}{} & \multicolumn{2}{c}{Times faster}\\
\cmidrule(lr){4-5}
$n$ & $m$ & $p$ & min & max \\
\midrule
20 & 20 & 20 & 0.71 & 2.47 \\
100 & 100 & 100 & 6.34 & 10.2 \\
400 & 400 & 400 & 19.2 & 36.8 \\
400 & 10 & 10 & 78.5 & 119.0\\ 
\bottomrule
\end{tabular}
\caption{
For each set of dimensions (given in the leftmost three columns), 
five different systems were randomly generated
and the running times to compute $\|G(\imagunit \omega)\|_2$ and \texttt{eig}($\MNcontnp$)
were recorded to form the ratios of these five pairs of values.  
Ratios greater than one indicate that it is faster to compute $\|G(\imagunit \omega)\|_2$ than
\texttt{eig}($\MNcontnp$), and by how much, while ratios less than one indicate
the opposite.  
The rightmost two columns of the table give the smallest and largest of the five ratios observed
per set of dimensions.
}
\label{table:ntf_vs_pencil}
\end{table}

\section{The improved algorithm}
\label{sec:hybrid_opt}
Recall that computing the $\Hinf$ norm is done by maximizing $\ntfc(\omega)$ over $\omega \in \R$ but
that the BBBS algorithm (and the cubic interpolation refinement) actually converges to a global maximum 
of $\ntfc(\omega)$ by iteratively computing the eigenvalues of the large matrix pencil $\MNpencont$ 
for successively 
larger values of $\gamma$.  However, we could alternatively consider a more direct approach of
finding maximizers of $\ntfc(\omega)$, which as discussed above, is a much cheaper function 
to evaluate numerically.  
Computing such maximizers could allow larger increases in $\gamma$ 
to be obtained on each iteration,
compared to just evaluating $\ntfc(\omega)$ at the midpoints or maximizers of the cubic interpolants.
This in turn should reduce the number of times that 
the eigenvalues of $\MNpencont$ must be computed and thus speed up the 
overall running time of the algorithm;
given the performance data in Table~\ref{table:ntf_vs_pencil}, the
additional cost of any evaluations of $\ntfc(\omega)$ needed to find maximizers
seems like it should be more than offset by fewer eigenvalue decompositions of 
 $\MNpencont$.
Of course, computing the eigenvalues of $\MNpencont$ at each iteration 
 cannot be eliminated completely, as it is still necessary for asserting
 whether or not any of the maximizers was a global maximizer 
 (in which case, the $\Hinf$ norm has been computed), or
 if not, to provide the remaining level set intervals where a global maximizer lies
 so the computation can continue.
 
As alluded to in the introduction, for the special case of the distance to instability, 
a similar cost-balancing strategy has been considered before in \cite{HeW99} but the authors themselves 
noted that the inverse iteration scheme they employed to find locally optimal solutions
could sometimes have very slow convergence and expressed concern that other optimization methods could suffer similarly.
Of course, in this paper we are considering the more general case of computing the $\Hinf$ norm and,
as we will observe in our later experimental evaluation, 
the first- and second-order optimization techniques we now propose  do in fact seem to work well in practice.

Though $\ntfc(\omega)$ is typically nonconvex, standard optimization methods should generally
still be able to find local maximizers, if not always global maximizers, provided that $\ntfc(\omega)$ is sufficiently smooth.  
Since $\ntfc(c)$ is the maximum singular value of $G(\imagunit\omega)$, it is locally Lipschitz 
(e.g. \cite[Corollary 8.6.2]{GolV13}).  
Furthermore, in proving the quadratic convergence of the midpoint-based BBBS algorithm, 
it was shown that at local maximizers, the second derivative of
$\ntfc(\omega)$ not only exists but is even locally Lipschitz \cite[Theorem~2.3]{BoyB90}.
The direct consequence is that
Newton's method for optimization can be expected to converge quadratically when it is used 
to find a local maximizer of $\ntfc(\omega)$.  Since there is only one optimization variable, namely $\omega$,
there is also the benefit that we need only work with first and second derivatives, 
instead of gradients and Hessians, respectively.  
Furthermore, if $\ntfc^{\prime\prime}(\omega)$ is expensive to compute,
one can instead resort to the secant method (which is a quasi-Newton method in one variable) 
and still obtain superlinear convergence.
Given the large disparity in costs to compute the eigenvalues of $\MNpencont
$ and $\ntfc(\omega)$, 
it seems likely that even just superlinear convergence could still be sufficient to significantly accelerate the computation of the $\Hinf$ norm.
Of course, when $\ntfc^{\prime\prime}(\omega)$ is relatively cheap to compute,
additional acceleration is likely to be obtained when using Newton's method.
Note that many alternative optimization strategies could also be employed here,
potentially with additional efficiencies. 
But, for sake of simplicity, we will just restrict the discussion in this paper to the secant method and 
Newton's method, particularly since conceptually there is no difference.

Since $\ntfc(\omega)$ will now need to be evaluated at any point requested by an optimization method, 
we will need to compute its first and possibly second derivatives directly; 
recall that using the eigenvectors of the purely imaginary eigenvalues of 
\eqref{eq:MNpencil_cont} with the equivalences in \eqref{eq:qs_vecs} and \eqref{eq:ntfcprime_eig_cont}
only allows us to obtain the first derivatives at the end points of the level-set intervals.
However, as long as we also compute the associated left and right singular vectors $u(\omega)$ and $v(\omega)$ when 
computing $\ntfc(\omega)$, the value of the first derivative $\ntfc^\prime(\omega)$ 
can be computed via the direct formulation given 
in \eqref{eq:ntfcprime_direct_cont} and without much additional cost over computing $\ntfc(\omega)$ itself.
For each frequency $\omega$ of interest,
an LU factorization of $(\imagunit \omega E - A)$ can be done once and reused to solve the 
linear systems due to the presence of $(\imagunit \omega E - A)^{-1}$, which appears 
once in $\ntfc(\omega)$ and twice in $\ntfc^\prime(\omega)$.

To compute $\ntfc^{\prime\prime}(\omega)$, we will need the following result for second derivatives of eigenvalues,
which can be found in various forms in \cite{Lan64}, \cite{OveW95}, and \cite{Kat82}.
\begin{theo}
\label{thm:eig2ndderiv}
For $t \in \R$, let $H(t)$ be a twice-differentiable $n \times n$ Hermitian matrix family with distinct eigenvalues at $t=0$ with $(\lambda_k,x_k)$ denoting the $k$th such eigenpair and where each eigenvector $x_k$ has unit norm
and the eigenvalues are ordered $\lambda_1 > \ldots > \lambda_n$.  
Then:
\[
	\lambda_1''(t) \bigg|_{t=0}= x_1^* H''(0) x_1 + 2 \sum_{k = 2}^{n} \frac{| x_1^* H'(0) x_k |^2}{\lambda_1 - \lambda_k}.
\]
\end{theo}
Since $\ntfc(\omega)$ is the largest singular value of $\|G(\imagunit \omega)\|_2$, it is also 
the largest eigenvalue of the matrix:
\beq
	\label{eq:eigderiv_mat}
	\eigderivmat(\omega) = 
	\begin{bmatrix} 0 & \tfc(\omega) \\ 
				\tfc(\omega)^* & 0 
	\end{bmatrix},
\eeq
which has first and second derivatives
\beq
	\label{eq:eigderiv12_mat}
	\eigderivmat^\prime(\omega) = 
	\begin{bmatrix} 0 & \tfc^\prime(\omega) \\ 
				\tfc^\prime(\omega)^* & 0
	\end{bmatrix}
	\quad \text{and} \quad
	\eigderivmat^{\prime\prime}(\omega) = 
	\begin{bmatrix} 0 & \tfc^{\prime\prime}(\omega) \\ 
				\tfc^{\prime\prime}(\omega)^* & 0
	\end{bmatrix}.
\eeq
The formula for $\tfc^\prime(\omega)$ is given by \eqref{eq:ntfcprime_direct_cont} while 
the corresponding second derivative is obtained by straightforward application
of matrix differentiation rules:
\beq
	\label{eq:tfc2_cont}
	\tfc^{\prime\prime}(\omega) = 
		-2 C(\imagunit \omega E - A)^{-1}E(\imagunit \omega E - A)^{-1}E(\imagunit \omega E - A)^{-1} B.
\eeq
Furthermore, the eigenvalues and eigenvectors of \eqref{eq:eigderiv_mat} needed 
to apply Theorem~\ref{thm:eig2ndderiv} are essentially directly available from just the full
SVD of $\tfc(\omega)$.  Let $\sigma_k$ be the $k$th singular value of $\tfc(\omega)$, along with associated
right and left singular vectors $u_k$ and $v_k$, respectively.  
Then $\pm\sigma_k$ is an eigenvalue of \eqref{eq:eigderiv_mat} with  
eigenvector $\left[ \begin{smallmatrix} u_k \\ v_k \end{smallmatrix} \right]$ for $\sigma_k$ and 
eigenvector $\left[ \begin{smallmatrix} u_k \\ -v_k \end{smallmatrix}\right]$ for $-\sigma_k$.
When $\sigma_k = 0$, the corresponding eigenvector is either 
$\left[\begin{smallmatrix} u_k \\ \mathbf{0} \end{smallmatrix} \right]$
if $p > m$ or
$\left[\begin{smallmatrix} \mathbf{0} \\ v_k \end{smallmatrix} \right]$
if $p < m$,
where $\mathbf{0}$ denotes a column of $m$ or $p$ zeros, respectively.
Given the full SVD of $\tfc(\omega)$,
computing $\ntfc^{\prime\prime}(\omega)$ can also be done with relatively little additional cost.
The stored LU factorization of $(\imagunit \omega E - A)$ used to obtain $\tfc(\omega)$ 
can again be reused to quickly compute the $\tfc^\prime(\omega)$ and $\tfc^{\prime\prime}(\omega)$
terms in \eqref{eq:eigderiv12_mat}.
If obtaining the full SVD is particularly expensive, i.e. for systems with many inputs/outputs, 
as mentioned above, sparse methods can still be used to efficiently obtain 
the largest singular value and its associated right/left singular vectors,
in order to at least calculate $\ntfc^{\prime}(\omega)$, if not $\ntfc^{\prime\prime}(\omega)$ as well.

\begin{rema}
On a more theoretical point, by invoking Theorem~\ref{thm:eig2ndderiv} to compute 
$\ntfc^{\prime\prime}(\omega)$, we are also assuming that the singular values of
$\tfc(\omega)$ are unique as well.  
However, in practice, this will almost certainly hold numerically, 
and to adversely impact the convergence rate of Newton's method,
it would have to frequently fail to hold,
which seems an exceptionally unlikely scenario.
As such, we feel that this additional assumption is not of practical concern.
\end{rema}

Thus, our new proposed improvement to the BBBS algorithm is to not settle for
the increase in $\gamma$ provided by the standard midpoint or cubic interpolation schemes, 
but to increase $\gamma$ \emph{as far as possible} on every iteration using standard optimization techniques 
applied to $\ntfc(\omega)$.  Assume that $\gamma$ is still less than the value of the $\Hinf$ norm and let 
\[
	\hat\omega_j = \argmax \ntfc(\hat \omega_k),
\]
where the finite set of $\hat\omega_k$ values are the midpoints of the level-set intervals $I_k$
or the maximizers of the cubic interpolants on these intervals, respectively
defined in \eqref{eq:gamma_mp} or \eqref{eq:gamma_cubic}.  Thus the solution
$\hat\omega_j \in I_j$ is the frequency that provides the updated value $\gamma_\mathrm{mp}$ 
or $\gamma_\mathrm{cubic}$ in the standard algorithms.
Now consider applying either Newton's method or the secant method (the choice of which one will be more efficient 
can be made more or less automatically depending on how $m,p$ compares to $n$)
to the following optimization 
problem with a simple box constraint:
\beq
	\label{eq:gamma_opt}
	\max_{\omega \in I_j} \ntfc(\omega).
\eeq
If the optimization method is initialized at $\hat\omega_j$,
then even if $\omega_\mathrm{opt}$, a \emph{computed} solution  to \eqref{eq:gamma_opt},
is actually just a \emph{local} maximizer 
(a possibility since \eqref{eq:gamma_opt} could be nonconvex),
it is still guaranteed that 
\[
	\ntfc(\omega_\mathrm{opt}) > 
		\begin{cases} 
		\gamma_\mathrm{mp} & \text{initial point $\hat \omega_j$ is a midpoint of $I_j$} \\
		\gamma_\mathrm{cubic} & \text{initial point $\hat \omega_j$ is a maximizer of interpolant $c_j(\omega)$}
		\end{cases}
\]
holds, provided that $\hat\omega_j$ does not happen to be a stationary point of $\ntfc(\omega)$.
Furthermore, \eqref{eq:gamma_opt} can only have more than one maximizer 
when the current estimate $\gamma$ of the $\Hinf$ norm is so low that 
there are multiple peaks above level-set interval $I_j$. 
Consequently,
as the algorithm converges, computed maximizers of \eqref{eq:gamma_opt} will be assured to be globally optimal over $I_j$
and in the limit, over all frequencies along the entire imaginary axis.
By setting tight tolerances for the optimization code, maximizers of 
\eqref{eq:gamma_opt} can also be computed to full precision with little to no penalty,
 due to the superlinear or quadratic rate of convergence 
we can expect from the secant method or Newton's method, respectively.
If the computed optimizer of \eqref{eq:gamma_opt} also happens to be a global maximizer 
of $\ntfc(\omega)$, for all $\omega \in \R$, 
then the $\Hinf$ norm has indeed been computed to full precision, but the algorithm must still verify this by 
computing the imaginary eigenvalues of $\MNpencont$ just one more time. 
However, if a global optimizer has not yet been found, then the algorithm must 
compute the imaginary eigenvalues of $\MNpencont$ at least two times more: one or more
times as the algorithm increases $\gamma$ to the globally optimal value, and then a final evaluation 
to verify that the computed value is indeed globally optimal.  Figure~\ref{fig:bbbs_opt}
shows a pictorial comparison of optimizing $\ntfc(\omega)$ compared to the midpoint and 
cubic-interpolant-based updating methods.

\begin{figure}
\subfloat[Optimization-based approach]{
\includegraphics[scale=.37,trim={1cm 0 2cm 0},clip]{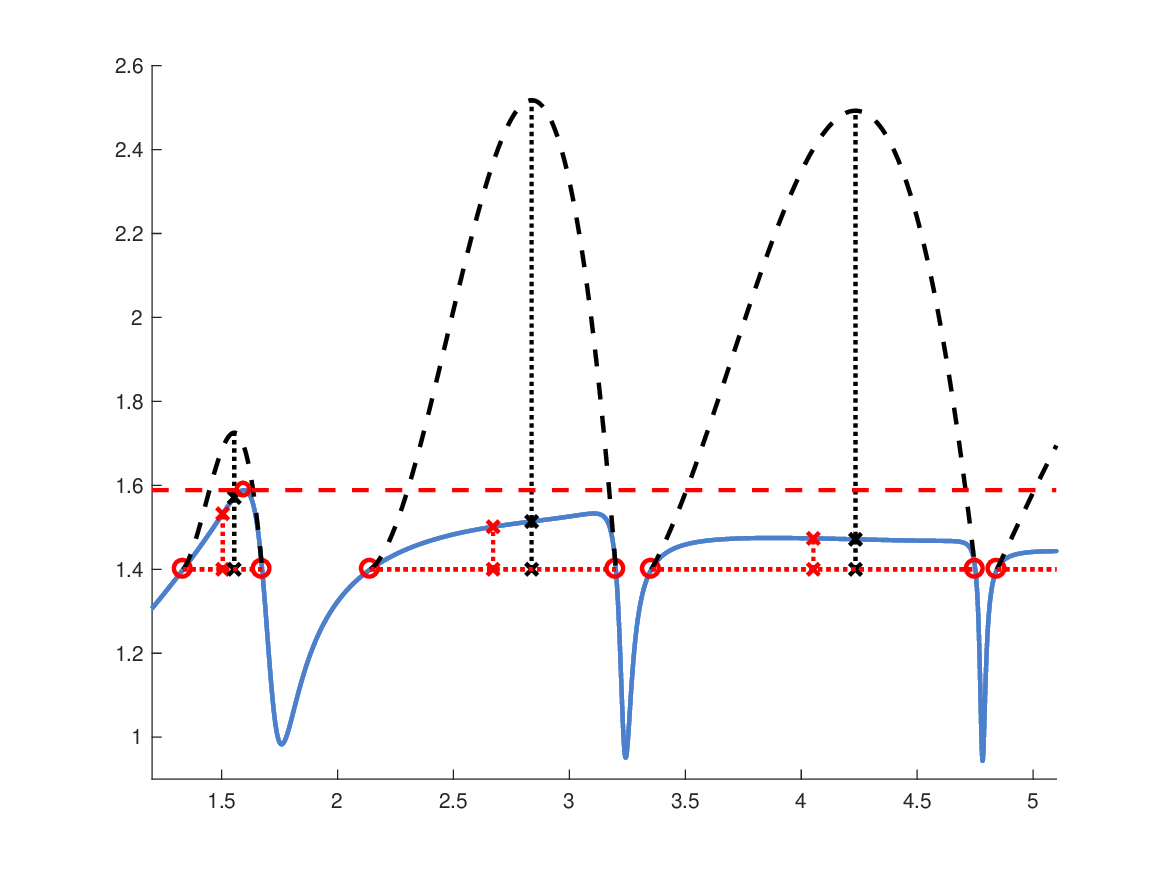} 
\label{fig:opt_comparison}
}
\subfloat[Close up of leftmost peak in Figure~\ref{fig:opt_comparison}]{
\includegraphics[scale=.37,trim={1cm 0 2cm 0},clip]{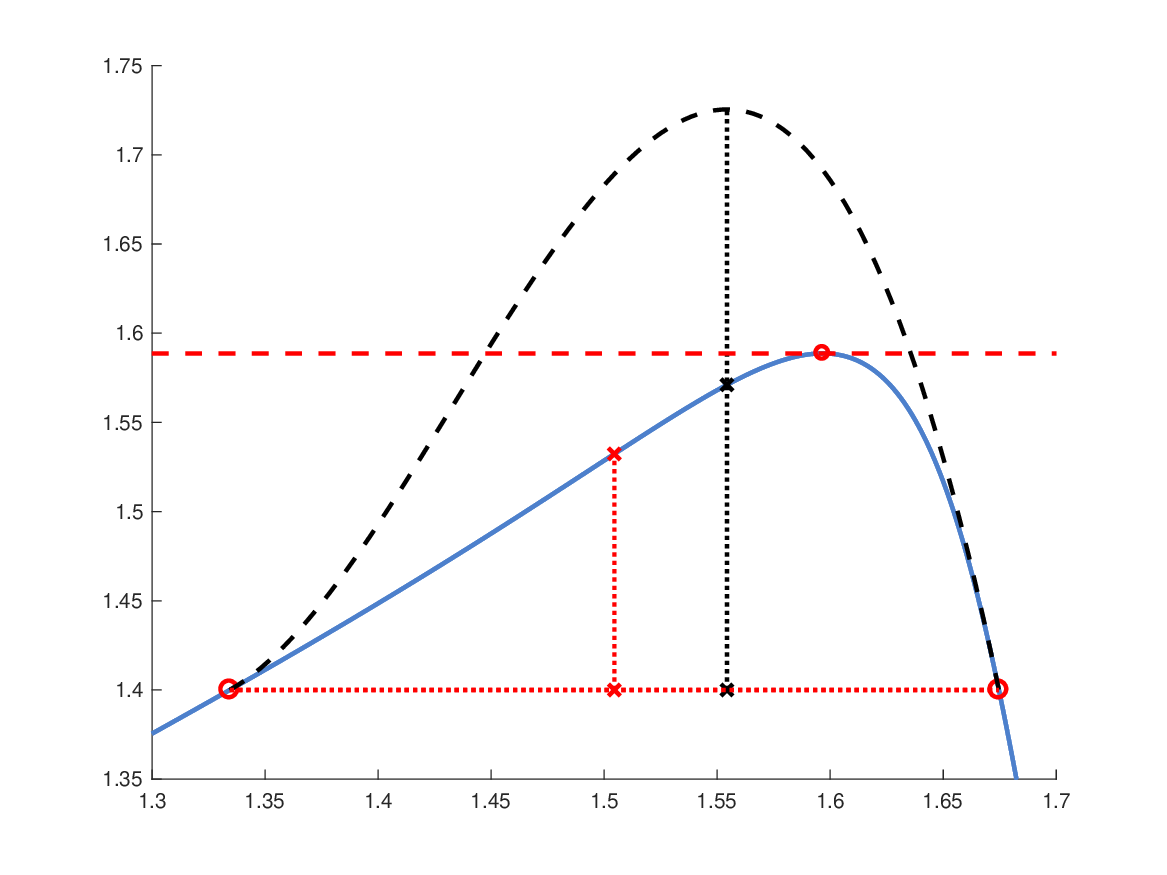}
}
\caption{For the same example as Figure~\ref{fig:bbbs_plots}, 
the larger increase in $\gamma$ attained by optimizing \eqref{eq:gamma_opt}
is shown by the red dashed line going through the red circle at the top of
the leftmost peak of $\| G(\imagunit \omega)\|_2$.  
By comparison, the BBBS midpoint (red dotted vertical lines and x's) and
the cubic-interpolant-based schemes (black dotted vertical lines and x's) 
only provide suboptimal increases in $\gamma$.}
\label{fig:bbbs_opt}
\end{figure}

In the above discussion, we have so far only considered applying optimization over 
the single level-set interval $I_j$ but we certainly could attempt to solve
\eqref{eq:gamma_opt} for other level-set intervals as well.
Let $\nopt = 1,2,\ldots$ be the max number of level-set intervals to optimize over per iteration
and $q$ be the number of level-set intervals for the current value of $\gamma$.
Compared to just optimizing over $I_j$, optimizing over all $q$ of the level-set intervals could yield
an even larger increase in the estimate $\gamma$ 
but would be the most expensive option computationally.
If we do not optimize over all the level-set intervals, i.e. $\nopt < q$,
there is the question of which intervals should be prioritized for optimization.
In our experiments, we found that 
prioritizing by
first evaluating $\ntfc(\omega)$ at all the $\hat\omega_k$ values and then choosing the intervals
$I_k$ to optimize where $\ntfc(\hat\omega_k)$ takes on the largest values
seems to be a good strategy. However,
with a serial MATLAB code, we have observed that just optimizing over the most promising interval,
i.e. $\nopt = 1$ so just $I_j$, 
is generally more efficient in terms of running time than trying to optimize over more intervals.
For discussion about optimizing over multiple intervals in parallel, see Section~\ref{sec:parallel}.

Finally, if a global maximizer of $\ntfc(\omega)$ can potentially be found before ever 
computing eigenvalues of $\MNpencont$ even once, 
then only one expensive eigenvalue decomposition $\MNpencont$ will be incurred,
just to verify that the initial maximizer is indeed a global one.
Thus, we also propose initializing the algorithm at a maximizer of $\ntfc(\omega)$,
 obtained via applying standard optimization techniques to
\beq
	\label{eq:gamma_init}
	\max \ntfc(\omega),
\eeq
which is just \eqref{eq:gamma_opt} without the box constraint.
In the best case, the computed maximizer will be a global one, but even a local
maximizer will still provide a higher initial estimate of the $\Hinf$ norm 
compared to initializing at a guess that may not even be locally optimal.
Of course, finding maximizers of \eqref{eq:gamma_init} by starting from multiple initial guesses can also be done in parallel; 
we again refer to Section~\ref{sec:parallel} for more details.

\begin{algfloat}
\begin{algorithm}[H]
\floatname{algorithm}{Algorithm}
\caption{The Improved Algorithm Using Local Optimization}
\label{alg:hybrid_opt}
\begin{algorithmic}[1]
	\REQUIRE{  
		Matrices $A \in \Cmn{n}{n}$, $B \in \Cmn{n}{m}$, $C \in \Cmn{p}{n}$, $D \in \Cmn{p}{m}$, 
		and $E \in \Cmn{n}{n}$, 
		initial frequency guesses $\{\omega_1,\ldots,\omega_q\} \in \R$ and
		$\nopt$, a positive integer indicating the number of intervals/frequencies to optimize per round.
	}
	\ENSURE{ 
		$\gamma = \| G \|_\hinfcal$ and $\omega$ such that 
		$\gamma = \ntfc(\omega)$.  
		\\ \quad
	}
	
	\STATE \COMMENT{Initialization:}
	\STATE Compute $[\gamma_1,\ldots,\gamma_q] = [\ntfc(\omega_1),\ldots,\ntfc(\omega_q)]$. 
			 \label{algline:init_start} 
	\STATE Reorder $\{\omega_1,\ldots,\omega_q\}$ s.t. $\gamma_j \ge \gamma_{j+1}$.
	\STATE Find maximal values $[\gamma_1,\ldots,\gamma_\nopt]$ of \eqref{eq:gamma_init}, using initial points
			$\omega_j$, $j=1,\ldots,\nopt$.  
			\label{algline:init_opt}
	\STATE $\gamma = \max([\gamma_1,\ldots,\gamma_\nopt])$.
			\label{algline:gamma_init}
	\STATE $\omega = \omega_j$ for $j$ s.t. $\gamma = \gamma_j$.
			\label{algline:init_end} 
	\STATE \COMMENT{Convergent Phase:}
	\WHILE {not converged} 
		\STATE \COMMENT{Compute the intervals that lie under $\ntfc(\omega)$ using 
						eigenvalues of the pencil:}
		\STATE Compute $\Lambda_\mathrm{I} =
				\{ \Im \lambda : \lambda \in \Lambda(\MNcontnp) ~\text{and}~ \Re \lambda = 0\}$.
				\label{algline:pencil}
		\STATE Index and sort $\Lambda_\mathrm{I} = \{\omega_1,...,\omega_l\}$ s.t. 
				$\omega_j \le \omega_{j+1}$.
				\label{algline:frequencies}
		\STATE Form all intervals $I_k = [\omega_k, \omega_{k+1}]$ s.t.
				each interval at height $\gamma$ is below $\ntfc(\omega)$.	
				\label{algline:hyopt_ints}
		\STATE \COMMENT{Compute candidate frequencies of the level-set intervals $I_k$ ($q$ of them):}
		\STATE Compute all $\hat\omega_k$ using either \eqref{eq:gamma_mp} or \eqref{eq:gamma_cubic}.
				\label{algline:hyopt_points}
		\STATE \COMMENT{Run box-constrained optimization on the $\nopt$ most promising frequencies:}
		\STATE Compute $[\gamma_1,\ldots,\gamma_q] = [\ntfc(\hat\omega_1),\ldots,\ntfc(\hat\omega_q)]$. 
		\STATE Reorder $\{\hat\omega_1,\ldots,\hat\omega_q\}$ and intervals $ I_k$ s.t. $\gamma_j \ge \gamma_{j+1}$.
			\label{algline:hyopt_reorder}
		\STATE Find maximal values $[\gamma_1,\ldots,\gamma_\nopt]$ of \eqref{eq:gamma_opt} using initial points $\hat\omega_j$, $j = 1,\ldots,\nopt$. 
		\label{algline:hyopt_opt}
		\STATE \COMMENT{Update to the highest gain computed:}
		\STATE $\gamma = \max([\gamma_1,\ldots,\gamma_\nopt])$.
		\STATE $\omega = \omega_j$ for $j$ s.t. $\gamma = \gamma_j$.
	\ENDWHILE
	\STATE \COMMENT{Check that the maximizing frequency is not at infinity (continuous-time only)}
	\IF { $\gamma < \|D\|_2$}
		\STATE $\gamma = \|D\|_2$.
		\STATE $\omega = \infty$.
	\ENDIF
\end{algorithmic}
\end{algorithm}
\algnote{
For details on how embarrassingly parallel processing can 
be used to further improve the algorithm, see Section~\ref{sec:parallel}.
}
\end{algfloat}

Algorithm~\ref{alg:hybrid_opt} provides a high-level pseudocode description of our improved method.
As a final step in the algorithm, it is necessary to check whether or not the value of the 
$\Hinf$ norm is attained at $\omega = \infty$.
We check this case after the convergent phase has computed a global maximizer over
the union of all intervals it has considered. 
The only possibility that the $\Hinf$ norm may be attained at $\omega = \infty$
in Algorithm~\ref{alg:hybrid_opt}  is when the initial value of $\gamma$ computed in 
line~\ref{algline:gamma_init} of is less than $\|D\|_2$.  As the assumptions of Theorem~\ref{thm:eigsing_cont} require 
that $\gamma$ not be a singular value of $D$, it is not valid to use $\gamma = \|D\|_2$ in 
$\MNpencont$ to check if this pencil has any imaginary eigenvalues.  However,
if the optimizer computed by the convergent phase of Algorithm~\ref{alg:hybrid_opt} yields a
$\gamma$ value less than $\|D\|_2$, then it is clear that the optimizing frequency is at $\omega = \infty$.

\section{Handling discrete-time systems}
\label{sec:discrete} Now consider the discrete-time linear dynamical system 
\begin{subequations}
\label{eq:lti_disc}
\begin{align}
Ex_{k+1} &=  Ax_k + Bu_k \\
y_k & =  Cx_k + Du_k,
\end{align}
\end{subequations}
where the matrices are defined as before in \eqref{eq:lti_cont}.  In this
case, the $\Hinf$ norm is defined as
\beq
	\label{eq:hinf_disc}
	\|G\|_{\hinfcal} \coloneqq \max_{\theta \in [0,2\pi)} \| G(\eitheta) \|_2,
\eeq
again assuming that pencil $(A,E)$ is at most index one.
If all finite eigenvalues are either strictly inside the unit disk centered 
at the origin or are uncontrollable or unobservable, then \eqref{eq:hinf_disc} is finite
and the $\Hinf$ norm is attained at some $\theta \in [0,2\pi)$.
Otherwise, it is infinite.

We now show the analogous version of Theorem~\ref{thm:eigsing_cont} for discrete-time systems.
The $D=0$ and $E=I$ case was considered in \cite[Section 3]{HinS91} 
while the more specific $B=C=E=I$ and $D=0$ case was given in 
\cite[Theorem 4]{Bye88}.\footnote{Note that equation (10) in \cite{Bye88} 
has a typo:  $A^H - \eitheta I$ in the lower left block of $K(\theta)$ should actually be 
$A^H - \eithetaconj I$.}
These results relate singular values of the transfer function for discrete-time systems
to eigenvalues with modulus one of associated matrix pencils.
Although the following more general result is already known,  
the proof, to the best of our knowledge, is not in the literature so we include it in full here.
The proof follows a similar argumentation as the proof of Theorem~\ref{thm:eigsing_cont}.
\begin{theo}
\label{thm:eigsing_disc}
Let $\lambda E - A$ be regular with no finite eigenvalues on the unit circle, $\gamma > 0$ not a singular value of $D$, and $\theta \in [0,2\pi)$.
Consider the matrix pencil $\MNpendisc$, where
\beq
	\label{eq:MNpencil_disc}
\begin{aligned}
	\Md \coloneqq {}& \begin{bmatrix} 	A - BR^{-1}D^*C 	& -\gamma BR^{-1}B^* \\ 
								0 				& E^* \end{bmatrix}, \\
	\Nd \coloneqq {}& \begin{bmatrix} 	E & 0\\ 
								-\gamma C^*S^{-1}C 	& (A - BR^{-1}D^*C)^*
								\end{bmatrix},
\end{aligned}
\eeq
$R = D^*D - \gamma^2 I$ and $S = DD^* - \gamma^2 I$.
Then $e^{\imagunit \theta}$ is an eigenvalue of matrix pencil $\MNpendisc$ if and only if
$\gamma$ is a singular value of $G(e^{\imagunit \theta})$. 
\end{theo}

\begin{proof}
Let $\gamma$ be a singular value of $G(\eitheta)$ with left and right singular vectors $u$ and $v$,
that is, so that $G(\eitheta)v = \gamma u$ and $G(\eitheta)^*u = \gamma v$.  
Using the expanded versions of these two equivalences 
\beq
	\label{eq:tfsv_equiv_disc}
	\left( \tfs{\eitheta} \right) v 	= \gamma u 
	\quad \text{and} \quad 
	\left( \tfs{\eitheta} \right)^* u  = \gamma v,
\eeq
we define 
\beq
	\label{eq:qs_disc}
	q = \left( \eitheta E - A \right)^{-1}Bv 
	\quad \text{and} \quad 
	s = \left( \eithetaconj E^* - A^* \right)^{-1}C^*u.
\eeq
Rewriting \eqref{eq:tfsv_equiv_disc} using \eqref{eq:qs_disc} yields the following matrix equation:
\beq
	\label{eq:uv_disc}
	\begin{bmatrix} C & 0 \\ 0 & B^* \end{bmatrix} 
	\begin{bmatrix} q \\ s \end{bmatrix} 
	= 
	\begin{bmatrix} -D & \gamma I \\ \gamma I & -D^* \end{bmatrix} 
	\begin{bmatrix} v \\ u \end{bmatrix} 
	~\Longrightarrow~
	\begin{bmatrix} v \\ u \end{bmatrix} 
	= 
	\begin{bmatrix} -D & \gamma I \\ \gamma I & -D^* \end{bmatrix}^{-1}
	\begin{bmatrix} C & 0 \\ 0 & B^* \end{bmatrix} 
	\begin{bmatrix} q \\ s \end{bmatrix} ,
\eeq
where
\beq
	\label{eq:Dgamma_inv}
	\begin{bmatrix} -D & \gamma I \\ \gamma I & -D^* \end{bmatrix}^{-1}
	= 
	\begin{bmatrix} -R^{-1}D^* & -\gamma R^{-1} \\ -\gamma S^{-1} & -DR^{-1} \end{bmatrix}
	\quad \text{and} \quad
	\begin{bmatrix} q \\ s \end{bmatrix} \ne 0.
\eeq
Rewriting \eqref{eq:qs_disc} as:
\beq
	\label{eq:EAqs1_disc}
	\left(
	\begin{bmatrix} \eitheta E & 0 \\ 0 & \eithetaconj E^* \end{bmatrix} -
	\begin{bmatrix} A & 0 \\ 0 & A^* \end{bmatrix}
	\right)
	\begin{bmatrix} q \\ s \end{bmatrix}
	= 
	\begin{bmatrix} B & 0 \\ 0 & C^* \end{bmatrix} 
	\begin{bmatrix} v \\ u \end{bmatrix} ,	
\eeq
and then substituting in \eqref{eq:uv_disc} for the rightmost term of \eqref{eq:EAqs1_disc} yields
\beq
	\label{eq:EAqs2_disc}
	\setlength\arraycolsep{3pt}
	\begingroup
	\left(
	\begin{bmatrix} \eitheta E & 0 \\ 0 & \eithetaconj E^* \end{bmatrix}  -
	\begin{bmatrix} A & 0 \\ 0 & A^* \end{bmatrix} 
	\right)
	\begin{bmatrix} q \\ s \end{bmatrix} 
	= 
	\begin{bmatrix} B & 0 \\ 0 & C^* \end{bmatrix} 
	\begin{bmatrix} -D & \gamma I \\ \gamma I & -D^* \end{bmatrix}^{-1}
	\begin{bmatrix} C & 0 \\ 0 & B^* \end{bmatrix} 
	\begin{bmatrix} q \\ s \end{bmatrix} .
	\endgroup
\eeq
Multiplying the above on the left by 
\[
	\begin{bmatrix} I & 0 \\ 0 & -\eitheta I \end{bmatrix}
\]
and then rearranging terms, we have
\beqs
	\label{eq:EAqs3_disc}
	\eitheta \begin{bmatrix} E & 0 \\ 0 & A^* \end{bmatrix}  
	\begin{bmatrix} q \\ s \end{bmatrix} 
	=
	\begin{bmatrix} A & 0 \\ 0 & E^* \end{bmatrix} 
	\begin{bmatrix} q \\ s \end{bmatrix} 
	+ 
	\begin{bmatrix} B & 0 \\ 0 & -\eitheta C^* \end{bmatrix}
	\begin{bmatrix} -D & \gamma I \\ \gamma I & -D^* \end{bmatrix}^{-1}
	\begin{bmatrix} C & 0 \\ 0 & B^* \end{bmatrix} 
	\begin{bmatrix} q \\ s \end{bmatrix} .
\eeqs
Substituting the inverse in \eqref{eq:Dgamma_inv} for its explicit form and multiplying terms yields:
\beqs
	\eitheta \begin{bmatrix} E & 0 \\ 0 & A^* \end{bmatrix}  
	\begin{bmatrix} q \\ s \end{bmatrix} 
	=
	\begin{bmatrix} A & 0 \\ 0 & E^* \end{bmatrix} 
	\begin{bmatrix} q \\ s \end{bmatrix} 
	+ 
	\begin{bmatrix} B & 0 \\ 0 & -\eitheta C^* \end{bmatrix}  
	\begin{bmatrix} -R^{-1}D^*C  & -\gamma R^{-1}B^* \\
		 -\gamma S^{-1}C & -DR^{-1}B^* \end{bmatrix} 
	\begin{bmatrix} q \\ s \end{bmatrix} .
\eeqs
Finally, multiplying terms further, separating out the $-\eitheta$ terms to bring them
over to the left hand side, and then recombining, we have that 
\beqs
	\eitheta \begin{bmatrix} E & 0 \\ -\gamma C^*S^{-1}C & A^* - C^*DR^{-1}B^* \end{bmatrix}  
	\begin{bmatrix} q \\ s \end{bmatrix} 
	=
	\begin{bmatrix} A -  BR^{-1}D^*C & -\gamma BR^{-1}B^* \\ 0 & E^* \end{bmatrix}
	\begin{bmatrix} q \\ s \end{bmatrix} .
\eeqs
It is now clear that $\eitheta$ is an eigenvalue of pencil $\MNpendisc$.

Now suppose that $\eitheta$ is an eigenvalue of pencil $\MNpendisc$ with eigenvector
given by $q$ and $s$ as above.  Then it follows that \eqref{eq:EAqs2_disc} holds, which can
be rewritten as \eqref{eq:EAqs1_disc} by defining $u$ and $v$ using the right-hand side equation of \eqref{eq:uv_disc}, noting that neither can be identically zero.  It is then clear that the pair of equivalences in \eqref{eq:qs_disc} hold.  Finally, substituting \eqref{eq:qs_disc} into the left-hand side equation of \eqref{eq:uv_disc}, it is clear that $\gamma$ is a singular value 
of $G(\eitheta)$, with left and right singular vectors $u$ and $v$.
\end{proof}

Adapting Algorithm~\ref{alg:hybrid_opt} to the discrete-time case is straightforward.  First, all instances of $\ntfc(\omega)$ 
must be replaced with 
\[
	\ntfd(\theta) \coloneqq \|G(\eitheta)\|_2.
\]
To calculate its first and second derivatives, we will need  
the first and second derivatives of $\tfd(\theta) \coloneqq G(\eitheta)$ and for notational brevity, it will
be convenient to define $Z(\theta) \coloneqq (\eitheta E - A)$.  Then
\beq
	\label{eq:tfd1st}
	\tfd^\prime(\theta) = - \imagunit \eitheta C Z(\theta)^{-1} E Z(\theta)^{-1} B
\eeq
and 
\beq
	\label{eq:tfd2nd}
	\tfd^{\prime\prime}(\theta) = 
		\eitheta C Z(\theta)^{-1} E Z(\theta)^{-1} B 
		- 2 e^{2\imagunit \theta} C Z(\theta)^{-1} E Z(\theta)^{-1} E Z(\theta)^{-1} B.
\eeq
The first derivative of $\ntfd(\theta)$ can thus be calculated using 
\eqref{eq:ntfcprime_defn_cont}, where $\omega$, $\ntfc(\omega)$, and $\tfc^\prime(\omega)$ are replaced
by $\theta$, $\ntfd(\theta)$, and $\tfd^\prime(\theta)$ using \eqref{eq:tfd1st}.  The second derivative
of $\ntfd(\theta)$ can be calculated using Theorem~\ref{thm:eig2ndderiv} using \eqref{eq:tfd1st}
and \eqref{eq:tfd2nd} to define $H(\theta)$, the analog of \eqref{eq:eigderiv_mat}.
Line~\ref{algline:pencil} must be changed
to instead compute the eigenvalues of unit modulus of \eqref{eq:MNpencil_disc}.  Line~\ref{algline:frequencies} must instead
index and sort the angles $\{\theta_1,\ldots,\theta_l\}$ of these unit modulus eigenvalues in ascending order.
Due to the periodic nature of \eqref{eq:hinf_disc}, line~\ref{algline:hyopt_ints} must additionally consider the 
``wrap-around" interval
$[\theta_l,\theta_0+2\pi]$.

\section{Numerical experiments}
\label{sec:numerical}
We implemented Algorithm~\ref{alg:hybrid_opt} in \matlab, for both continuous-time and discrete-time cases.  
Since we can only get timing information from \texttt{hinfnorm} and we wished to verify that our new method
does indeed reduce the number of times the eigenvalues of $\MNpencont$ 
and $\MNpendisc$ are computed, 
we also designed our code so that it can run just using the standard BBBS algorithm or the cubic-interpolant scheme.  
For our new optimization-based approach, we used \texttt{fmincon} for both the unconstrained optimization 
calls needed for the initialization phase and for the box-constrained optimization calls needed in the convergent phase;
\texttt{fmincon}'s optimality and constraint tolerances were set to $10^{-14}$ in order to find maximizers to near machine
precision. Our code supports starting the latter optimization calls from either the midpoints of the BBBS algorithm 
\eqref{eq:gamma_mp} or the maximizing frequencies calculated from the cubic-interpolant method 
\eqref{eq:gamma_cubic}.  Furthermore, the optimizations may be done using either
 the secant method (first-order information only) or with Newton's method using second derivatives, thus leading 
 to four variants of our proposed method to test.  Our code has a user-settable parameter that determines 
 when $m,p$ should be considered too large relative to $n$, and thus when it is likely that using secant method  
 will actually be faster than Newton's method, due to the additional expense of computing the second derivative
 of the norm of the transfer function.
  
For initial frequency guesses, our code simply tests zero and the imaginary part of the 
rightmost eigenvalue of $(A,E)$, excluding eigenvalues that are either infinite, uncontrollable, or unobservable.  
Eigenvalues are deemed uncontrollable or unobservable if $\|B^*y\|_2$ or $\|Cx\|_2$ are respectively below a user-set 
tolerance, where $x$ and $y$ are respectively the right and left eigenvectors for a given eigenvalue of $(A,E)$.  
In the discrete-time case, the default initial guesses are zero, $\pi$, and the angle for the largest modulus eigenvalue.\footnote{
For producing a production-quality implementation, see \cite{BruS90} for more sophisticated initial guesses that can be used, 
\cite[Section III]{BenV11} for dealing with testing properness of the transfer function, and \cite{Var90a} for filtering out 
uncontrollable/unobservable eigenvalues of $(A,E)$ when it has index higher than one.}
 
For efficiency of implementing our method and conducting these experiments, 
our code does not yet take advantage of structure-preserving eigensolvers.  Instead,
it uses the regular QZ algorithm (\texttt{eig} in \matlab) to compute the eigenvalues of $\MNpencont$ and $\MNpendisc$.
To help mitigate issues due to rounding errors, we 
consider any eigenvalue $\lambda$ imaginary or of unit modulus
if it lies within a margin of width $10^{-8}$ on either side of the 
imaginary axis or unit circle.
Taking the imaginary parts of these nearly imaginary eigenvalues forms the 
initial set of candidate frequencies, or the angles of these nearly unit modulus eigenvalues for the discrete-time case.  
Then we simply form all the consecutive intervals, including the wrap-around interval for the discrete-time case, even though not all of them will be level-set intervals, and some intervals may only be a portion of a level-set interval (e.g. if the use of QZ causes spurious candidate frequencies).
The reason we do this is is because we can easily sort which of the intervals at height $\gamma$ 
are below $\ntfc(\omega)$ or $\ntfd(\omega)$ just by evaluating these functions at 
the midpoint or the maximizer of the cubic interpolant for each interval.
This is less expensive because we need to evaluate these interior points regardless,
so also evaluating the norm of the transfer function at all these endpoints just
adds additional cost.
However, for the cubic interpolant refinement, we nonetheless still evaluate 
$\ntfc(\omega)$ or $\ntfd(\omega)$ at the endpoints since we need the corresponding derivatives there
to construct the cubic interpolants; we do not use the eigenvectors 
of $\MNpencont$ or $\MNpendisc$
to bypass this additional cost as \texttt{eig} in \matlab\ does not 
currently provide a way to only compute selected eigenvectors, i.e. those corresponding 
to the imaginary (unit-modulus) eigenvalues.
Note that while this strategy is sufficient for our experimental comparisons here, 
it certainly does not negate the need for structure-preserving eigenvalue solvers.

We evaluated our code on several continuous- and discrete-time problems up to 
moderate dimensions, all listed with dimensions in Table~\ref{table:hinfnorm_tol}.
For the continuous-time problems, we chose four problems 
from the test set used in \cite{GugGO13} (\texttt{CBM}, \texttt{CSE2}, \texttt{CM3}, \texttt{CM4}),
two from the SLICOT benchmark examples\footnote{Available at \url{http://slicot.org/20-site/126-benchmark-examples-for-model-reduction}}
(\texttt{ISS} and \texttt{FOM}), and 
two new randomly-generated examples using \texttt{randn()}
with a relatively large number of inputs and outputs. 
Notably, the four problems from \cite{GugGO13} were generated
via taking open-loop systems from \compleib\ \cite{compleib}
and then designing controllers to minimize the $\Hinf$ norm of the 
corresponding closed-loop systems via \hifoo\ \cite{BurHLetal06}.
Such systems can be interesting benchmark examples
because $\ntfc(\omega)$ will often have several peaks,
and multiple peaks may attain the value of the $\Hinf$ norm,
or at least be similar in height.
Since the discrete-time problems from \cite{GugGO13} were 
all very small scale (the largest order in that test set is only 16) and SLICOT only offers a single discrete-time benchmark example,
we instead elected to take additional open-loop systems from \compleib\ and obtain usable 
test examples by minimizing the discrete-time $\Hinf$ norm of their respective closed-loop 
systems, via optimizing controllers using \hifood\ \cite{PopWM10}, 
a fork of \hifoo\ for discrete-time systems.
On all examples, the $\Hinf$ norm values computed by our local-optimization-enhanced code (in all its variants) agreed on average to 13 digits with the results provided by \texttt{hinfnorm}, when used with the tight tolerance of $10^{-14}$, with the worst discrepancy being only 11 digits of agreement.
However, our improved method often found slightly larger values, i.e. more accurate values, since it
optimizes $\ntfc(\omega)$ and $\ntfd(\omega)$ directly.

All experiments were performed using \matlab\ R2016b running on a Macbook Pro with an Intel i7-6567U dual-core CPU, 16GB of RAM, and Mac OS X v10.12.

\subsection{Continuous-time examples}

\begin{table}[!t]
\center
\begin{tabular}{ l | cccc | cc } 
\toprule
\multicolumn{7}{c}{Small-scale examples (continuous time)}\\
\midrule
\multicolumn{1}{c}{} & 
	\multicolumn{4}{c}{Hybrid Optimization} & \multicolumn{2}{c}{Standard Algs.} \\
\cmidrule(lr){2-5}
\cmidrule(lr){6-7}
\multicolumn{1}{c}{} &
	\multicolumn{2}{c}{Newton} & \multicolumn{2}{c}{Secant} & \multicolumn{2}{c}{} \\
\cmidrule(lr){2-3}
\cmidrule(lr){4-5}
\multicolumn{1}{l}{Problem} &
	\multicolumn{1}{c}{Interp.} & \multicolumn{1}{c}{MP} & \multicolumn{1}{c}{Interp.} & \multicolumn{1}{c}{MP} &
    	\multicolumn{1}{c}{Interp.} & \multicolumn{1}{c}{BBBS} \\
\midrule	
\multicolumn{7}{c}{Number of Eigenvalue Computations of $\MNpencont$}\\
\midrule
\texttt{CSE2}    & 2 & 3 & 1 & 1 & 2 & 3 \\ 
\texttt{CM3}     & 2 & 3 & 2 & 2 & 3 & 5 \\ 
\texttt{CM4}     & 2 & 2 & 2 & 2 & 4 & 6 \\ 
\texttt{ISS}     & 1 & 1 & 1 & 1 & 3 & 4 \\ 
\texttt{CBM}     & 2 & 2 & 2 & 2 & 5 & 7 \\ 
\texttt{randn 1}   & 1 & 1 & 1 & 1 & 1 & 1 \\ 
\texttt{randn 2}   & 1 & 1 & 1 & 1 & 2 & 2 \\ 
\texttt{FOM}     & 1 & 1 & 1 & 1 & 2 & 2 \\ 
\midrule	
\multicolumn{7}{c}{Number of Evaluations of $\ntfc(\omega)$}\\
\midrule
\texttt{CSE2}    & 10 & 7 & 10 & 10 & 9 & 8 \\ 
\texttt{CM3}     & 31 & 26 & 53 & 45 & 31 & 24 \\ 
\texttt{CM4}     & 19 & 17 & 44 & 43 & 46 & 36 \\ 
\texttt{ISS}     & 12 & 12 & 22 & 22 & 39 & 27 \\ 
\texttt{CBM}     & 34 & 28 & 59 & 55 & 46 & 36 \\ 
\texttt{randn 1}   & 1 & 1 & 1 & 1 & 1 & 1 \\ 
\texttt{randn 2}   & 4 & 4 & 17 & 17 & 6 & 4 \\ 
\texttt{FOM}     & 4 & 4 & 16 & 16 & 7 & 5 \\ 
\bottomrule
\end{tabular}
\caption{The top half of the table reports the number of times the eigenvalues of 
$\MNpencont$ were computed in order to compute the $\Hinf$ norm
to near machine precision.  From left to right, the methods are our hybrid optimization
approach using Newton's method and the secant method, 
the cubic-interpolant scheme (column `Interp')
and the standard BBBS method, all implemented by our single configurable code.
The subcolumns `Interp.' and 'MP' of our methods respectively indicate that the optimization 
routines were initialized at the points from the cubic-interpolant scheme and the BBBS midpoint scheme.
The bottom half of the table reports the number of times it was necessary to evaluate the
norm of the transfer function (with or without its derivatives).
The problems are listed in increasing order of their state-space sizes $n$; 
for their exact dimensions, see Table~\ref{table:hinfnorm_tol}.
}
\label{table:evals_dense}
\end{table}

In Table~\ref{table:evals_dense}, we list the number of times the eigenvalues of
$\MNpencont$ were computed and the number of evaluations of $\ntfc(\omega)$ for our 
new method compared to our implementations of the existing BBBS algorithm and its interpolation-based refinement.  
As can be seen, our new method typically limited the number of required eigenvalue computations of 
$\MNpencont$ to just two, and often it only required one (in the cases where our method 
found a global optimizer of $\ntfc(\omega)$ in the initialization phase).  In contrast, the standard 
BBBS algorithm and its interpolation-based refinement had to evaluate the eigenvalues 
$\MNpencont$ more times; for example, on problem
\texttt{CBM}, the BBBS algorithm needed seven evaluations while its interpolation-based refinement still needed five.
Though our new method sometimes required more evaluations of $\ntfc(\omega)$ 
than the standard algorithms, often the number of evaluations of $\ntfc(\omega)$
was actually less with our new method, presumably due its fewer iterations and 
particularly when using the Newton's method variants.
Even when our method required more evaluations of $\ntfc(\omega)$ than the standard methods,
the increases were not too significant (e.g. the secant method variants of our method on problems \texttt{CM4}, \texttt{CBM}, \texttt{randn 2}, and \texttt{FOM}).
Indeed, the larger number of evaluations of $\ntfc(\omega)$
when employing the secant method in lieu of Newton's method
was still generally quite low.

\begin{table}[!t]
\setlength{\tabcolsep}{3pt}
\robustify\bfseries
\center
\begin{tabular}{ l | SSSS | SS | SS } 
\toprule
\multicolumn{9}{c}{Small-scale examples (continuous time)} \\
\midrule
\multicolumn{1}{c}{} & 
	\multicolumn{4}{c}{Hybrid Optimization} & \multicolumn{2}{c}{Standard Algs.} & \multicolumn{2}{c}{\texttt{hinfnorm($\cdot$,tol)}}\\
\cmidrule(lr){2-5}
\cmidrule(lr){6-7}
\cmidrule(lr){8-9}
\multicolumn{1}{c}{} &
	\multicolumn{2}{c}{Newton} & \multicolumn{2}{c}{Secant} & \multicolumn{2}{c}{} & \multicolumn{2}{c}{\texttt{tol}}\\
\cmidrule(lr){2-3}
\cmidrule(lr){4-5}
\cmidrule(lr){8-9}
\multicolumn{1}{l}{Problem} &
	\multicolumn{1}{c}{Interp.} & \multicolumn{1}{c}{MP} & \multicolumn{1}{c}{Interp.} & \multicolumn{1}{c}{MP} &
    	\multicolumn{1}{c}{Interp.} & \multicolumn{1}{c}{BBBS} &
    	\multicolumn{1}{c}{\texttt{1e-14}} & \multicolumn{1}{c}{\texttt{0.01}} \\
\midrule
\multicolumn{9}{c}{Wall-clock running times in seconds}\\
\midrule
\texttt{CSE2}    & 0.042 & 0.060 & 0.036 & 0.032 & 0.043 & \bfseries 0.031 & 0.137 & 0.022 \\ 
\texttt{CM3}     & \bfseries 0.125 & 0.190 & 0.198 & 0.167 & 0.170 & 0.167 & 0.148 & 0.049 \\ 
\texttt{CM4}     & \bfseries 0.318 & 0.415 & 0.540 & 0.550 & 0.712 & 0.811 & 1.645 & 0.695 \\ 
\texttt{ISS}     & 0.316 & 0.328 & 0.379 & \bfseries 0.303 & 0.709 & 0.757 & 0.765 & 0.391 \\ 
\texttt{CBM}     & 0.744 & \bfseries 0.671 & 1.102 & 1.071 & 1.649 & 1.757 & 3.165 & 1.532 \\ 
\texttt{randn 1}   & 0.771 & 0.868 & 1.006 & 0.871 & \bfseries 0.700 & 0.756 & 21.084 & 30.049 \\ 
\texttt{randn 2}   & \bfseries 9.551 & 9.746 & 9.953 & 11.275 & 14.645 & 15.939 & 31.728 & 16.199 \\ 
\texttt{FOM}     & \bfseries 3.039 & 3.426 & 4.418 & 4.176 & 5.509 & 5.182 & 128.397 & 36.529 \\ 
\midrule
\multicolumn{9}{c}{Running times relative to hybrid optimization (Newton with `Interp.')}\\
\midrule
\texttt{CSE2}    & 1 & 1.42 & 0.86 & 0.75 & 1.02 & 0.75 & 3.24 & 0.53 \\ 
\texttt{CM3}     & 1 & 1.52 & 1.59 & 1.34 & 1.36 & 1.34 & 1.19 & 0.39 \\ 
\texttt{CM4}     & 1 & 1.31 & 1.70 & 1.73 & 2.24 & 2.55 & 5.18 & 2.19 \\ 
\texttt{ISS}     & 1 & 1.04 & 1.20 & 0.96 & 2.24 & 2.39 & 2.42 & 1.24 \\ 
\texttt{CBM}     & 1 & 0.90 & 1.48 & 1.44 & 2.22 & 2.36 & 4.26 & 2.06 \\ 
\texttt{randn 1}   & 1 & 1.13 & 1.31 & 1.13 & 0.91 & 0.98 & 27.36 & 38.99 \\ 
\texttt{randn 2}   & 1 & 1.02 & 1.04 & 1.18 & 1.53 & 1.67 & 3.32 & 1.70 \\ 
\texttt{FOM}     & 1 & 1.13 & 1.45 & 1.37 & 1.81 & 1.71 & 42.25 & 12.02 \\ 
\midrule
Average          & 1 & 1.18 & 1.33 & 1.24 & 1.67 & 1.72 & 11.15 & 7.39 \\
\bottomrule
\end{tabular}
\caption{
In the top half of the table, the running times (fastest in bold) are reported 
in seconds for the 
same methods and configurations as in Table~\ref{table:evals_dense}, with
the running times of \texttt{hinfnorm} additionally listed in the rightmost two
columns, for a tolerance of $10^{-14}$ (as used by the other methods) and its default 
value of $0.01$.
The bottom of half the table normalizes all the times relative to the running times for our 
hybrid optimization method (Newton and `Interp.'), 
along with the overall averages relative to this variant.
}
\label{table:times_dense}
\end{table}

In Table~\ref{table:times_dense}, we compare the corresponding wall-clock times, and for convenience, 
we replicate the timing results of \texttt{hinfnorm} from Table~\ref{table:hinfnorm_tol} on
the same problems.
We observe that our new method was fastest on six out of the eight test problems, often significantly so. 
Compared to our own implementation of the BBBS algorithm, our new method was on average 1.72 times as fast and on three problems, 2.36-2.55 times faster.  
We see similar speedups compared to the cubic-interpolation refinement method as well.
Our method was even faster when compared to \texttt{hinfnorm}, which had the advantage 
of being a compiled code rather than interpreted like our code.  
Our new method was over eleven times faster than \texttt{hinfnorm} overall, but this was largely due to the two problems (\texttt{FOM} and \texttt{randn 1})
where our code was 27-42 times faster.  We suspect that this large performance gap on these problems 
was not necessarily due to
a correspondingly dramatic reduction in the number of times that the eigenvalues of $\MNpencont$
were computed but rather that the structure-preserving eigensolver \texttt{hinfnorm} employed sometimes 
has a steep performance penalty compared to standard QZ.  However, it is difficult to verify this as \texttt{hinfnorm}
is not open source.  We also see that for the variants of our method, there was about a 24-33\% percent penalty on average  
in the runtime when resorting to the secant method instead of Newton's method.
Nonetheless, even the slower secant-method-based 
version of our hybrid optimization approach was still typically much faster than BBBS or the cubic-interpolation scheme.   
The only exception to this was problem \texttt{CSE2}, where our secant method variants were
actually faster than our Newton's method variants; the reason for this was because during initialization,
the Newton's method optimization just happened to find worse initial local maximizers than 
the secant method approach, which led to more eigenvalue computations of $\MNpencont$.

The two problems where the variants of our new method were not fastest
were \texttt{CSE2} and \texttt{randn 1}.  However, for \texttt{CSE2},
our secant method variant using midpoints was essentially as fast as the standard algorithm.
As mentioned above, the Newton's method variants ended up being slower since 
they found worse initial local maximizers.
For \texttt{randn 1}, all methods only required a single evaluation of $\ntfc(\omega)$ and computing the eigenvalues of $\MNpencont$; in other words, their respective initial guesses were all actually a global maximizer.
As such, the differences in running times for \texttt{randn 1} seems likely attributed to the variability of interpreted \matlab\ code.

\subsection{Discrete-time examples}
We now present corresponding experiments for the six discrete-time examples listed in Table~\ref{table:hinfnorm_tol}.
In Table~\ref{table:evals_dense_disc},
we see that our new approach on discrete-time problems also reduces the number of expensive eigenvalue computations of $\MNpendisc$ compared to the standard methods
and that in the worst cases, there is only moderate increase in the number of evaluations of $\ntfd(\omega)$
and often, even a reduction, similarly as we saw in Table~\ref{table:evals_dense} for the continuous-time problems.

Wall-clock running times are reported in Table~\ref{table:times_dense_disc}, and show similar results, if not identical, to those in Table~\ref{table:times_dense} for the continuous-time comparison.
We see that our Newton's method variants are, on average, 1.66 and 1.41 times faster, respectively,
than the BBBS and cubic-interpolation refinement algorithms.
Our algorithms are often up to two times faster than these two standard methods
and were even up to 25.2 times faster on \texttt{ISS1d} compared to \texttt{hinfnorm} using \texttt{tol=1e-14}.
For three of the six problems, our approach was not fastest but these
three problems (\texttt{LAHd}, \texttt{BDT2d}, \texttt{EB6d}) 
also had the smallest orders among the discrete-time examples ($n=58,92,170$, respectively).
This underscores that our approach is likely most beneficial for all but rather small-scale problems,
where there is generally an insufficient cost gap between computing $\ntfd(\omega)$ 
and the eigenvalues of $\MNpendisc$.
However, for \texttt{LAHd} and \texttt{EB6d}, it was actually \texttt{hinfnorm}
that was fastest, where we are comparing a compiled 
code to our own pure \matlab\ interpreted code.
Furthermore, on these two problems, our approach was nevertheless not dramatically slower
than \texttt{hinfnorm} and for \texttt{EB6d}, was actually faster than our own implementation 
of the standard algorithms.
Finally, on \texttt{BDT2}, the fastest version of our approach essentially matched
the performance of our BBBS implementation, if not the cubic-interpolation refinement.

\begin{table}[!t]
\center
\begin{tabular}{ l | cccc | cc } 
\toprule
\multicolumn{7}{c}{Small-scale examples (discrete time)} \\
\midrule
\multicolumn{1}{c}{} & 
	\multicolumn{4}{c}{Hybrid Optimization} & \multicolumn{2}{c}{Standard Algs.} \\
\cmidrule(lr){2-5}
\cmidrule(lr){6-7}
\multicolumn{1}{c}{} &
	\multicolumn{2}{c}{Newton} & \multicolumn{2}{c}{Secant} & \multicolumn{2}{c}{} \\
\cmidrule(lr){2-3}
\cmidrule(lr){4-5}
\multicolumn{1}{l}{Problem} &
	\multicolumn{1}{c}{Interp.} & \multicolumn{1}{c}{MP} & \multicolumn{1}{c}{Interp.} & \multicolumn{1}{c}{MP} &
    	\multicolumn{1}{c}{Interp.} & \multicolumn{1}{c}{BBBS} \\
\midrule	
\multicolumn{7}{c}{Number of Eigenvalue Computations of $\MNpendisc$}\\
\midrule
\texttt{LAHd}    & 2 & 2 & 1 & 1 & 3 & 4 \\ 
\texttt{BDT2d}   & 2 & 3 & 2 & 2 & 3 & 4 \\ 
\texttt{EB6d}    & 1 & 1 & 2 & 1 & 3 & 5 \\ 
\texttt{ISS1d}   & 1 & 1 & 1 & 1 & 2 & 2 \\ 
\texttt{CBMd}    & 1 & 1 & 1 & 1 & 3 & 2 \\ 
\texttt{CM5d}    & 2 & 2 & 2 & 2 & 3 & 4 \\ 
\midrule	
\multicolumn{7}{c}{Number of Evaluations of $\ntfd(\omega)$}\\
\midrule
\texttt{LAHd}    & 13 & 11 & 24 & 24 & 17 & 15 \\ 
\texttt{BDT2d}   & 17 & 18 & 43 & 40 & 18 & 17 \\  
\texttt{EB6d}    & 22 & 22 & 37 & 34 & 32 & 32 \\ 
\texttt{ISS1d}   & 5 & 5 & 24 & 24 & 7 & 6 \\ 
\texttt{CBMd}    & 5 & 5 & 26 & 26 & 12 & 6 \\ 
\texttt{CM5d}    & 20 & 16 & 27 & 27 & 22 & 18 \\ 
\bottomrule
\end{tabular}
\caption{The column headers remain as described for Table~\ref{table:evals_dense}.
}
\label{table:evals_dense_disc}
\end{table}

\begin{table}[!t]
\setlength{\tabcolsep}{3pt}
\robustify\bfseries
\center
\begin{tabular}{ l | SSSS | SS | SS } 
\toprule
\multicolumn{9}{c}{Small-scale examples (discrete time)} \\
\midrule
\multicolumn{1}{c}{} & 
	\multicolumn{4}{c}{Hybrid Optimization} & \multicolumn{2}{c}{Standard Algs.} & \multicolumn{2}{c}{\texttt{hinfnorm($\cdot$,tol)}}\\
\cmidrule(lr){2-5}
\cmidrule(lr){6-7}
\cmidrule(lr){8-9}
\multicolumn{1}{c}{} &
	\multicolumn{2}{c}{Newton} & \multicolumn{2}{c}{Secant} & \multicolumn{2}{c}{} & \multicolumn{2}{c}{\texttt{tol}}\\
\cmidrule(lr){2-3}
\cmidrule(lr){4-5}
\cmidrule(lr){8-9}
\multicolumn{1}{l}{Problem} &
	\multicolumn{1}{c}{Interp.} & \multicolumn{1}{c}{MP} & \multicolumn{1}{c}{Interp.} & \multicolumn{1}{c}{MP} &
    	\multicolumn{1}{c}{Interp.} & \multicolumn{1}{c}{BBBS} &
    	\multicolumn{1}{c}{\texttt{1e-14}} & \multicolumn{1}{c}{\texttt{0.01}} \\
\midrule
\multicolumn{9}{c}{Wall-clock running times in seconds}\\
\midrule
\texttt{LAHd}    & 0.051 & 0.038 & 0.056 & 0.056 & 0.034 & 0.040 & \bfseries 0.031 & 0.015 \\ 
\texttt{BDT2d}   & 0.075 & 0.123 & 0.146 & 0.191 & \bfseries 0.057 & 0.076 & 0.070 & 0.031 \\ 
\texttt{EB6d}    & 0.271 & 0.312 & 0.409 & 0.296 & 0.469 & 0.732 & \bfseries 0.192 & 0.122 \\ 
\texttt{ISS1d}   & 0.654 & \bfseries 0.636 & 0.828 & 0.880 & 1.168 & 1.291 & 16.495 & 3.930 \\ 
\texttt{CBMd}    & 0.898 & \bfseries 0.795 & 0.999 & 1.640 & 2.015 & 1.420 & 1.411 & 0.773 \\ 
\texttt{CM5d}    & \bfseries 7.502 & 8.022 & 9.887 & 8.391 & 9.458 & 14.207 & 10.802 & 2.966 \\ 
\midrule
\multicolumn{9}{c}{Running times relative to hybrid optimization (Newton with `Interp.')}\\
\midrule
\texttt{LAHd}    & 1 & 0.74 & 1.09 & 1.11 & 0.67 & 0.78 & 0.60 & 0.29 \\ 
\texttt{BDT2d}   & 1 & 1.65 & 1.96 & 2.55 & 0.76 & 1.01 & 0.93 & 0.41 \\ 
\texttt{EB6d}    & 1 & 1.15 & 1.51 & 1.09 & 1.73 & 2.70 & 0.71 & 0.45 \\ 
\texttt{ISS1d}   & 1 & 0.97 & 1.27 & 1.34 & 1.79 & 1.97 & 25.21 & 6.01 \\ 
\texttt{CBMd}    & 1 & 0.89 & 1.11 & 1.83 & 2.24 & 1.58 & 1.57 & 0.86 \\ 
\texttt{CM5d}    & 1 & 1.07 & 1.32 & 1.12 & 1.26 & 1.89 & 1.44 & 0.40 \\ 
\midrule
Avg.         & 1 & 1.08 & 1.38 & 1.51 & 1.41 & 1.66 & 5.08 & 1.40 \\ 
\bottomrule
\end{tabular}
\caption{The column headers remain as described for Table~\ref{table:times_dense}.
}
\label{table:times_dense_disc}
\end{table}

\section{Local optimization for $\Hinf$ norm approximation}
\label{sec:hinf_approx}
Unfortunately, the $\bigO(n^3)$ work necessary to compute all the imaginary eigenvalues of $\MNpencont$
restricts the usage of the level-set ideas from \cite{Bye88,BoyBK89} to rather small-dimensional problems.
The same computational limitation of course also holds for obtaining all of the unit-modulus eigenvalues of $\MNpendisc$
in the discrete-time case.
Currently there is no known alternative technique that would guarantee convergence to a global maximizer of $\ntfc(\omega)$ 
or $\ntfd(\theta)$, to thus ensure exact computation of the $\Hinf$ norm, while also having more favorable scaling properties.
Indeed, the aforementioned scalable methods of \cite{GugGO13,BenV14,FreSV14,MitO16,AliBMetal17} 
for approximating the $\Hinf$ norm of large-scale systems
all forgo the expensive operation of computing all the eigenvalues of $\MNpencont$ and $\MNpendisc$, 
and consequently, the most that any of them can guarantee in terms of accuracy
is that they converge to a local maximizer of $\ntfc(\omega)$ or $\ntfd(\theta)$.
However, a direct consequence of our work here to accelerate the exact computation 
of the $\Hinf$ norm is that the straightforward application of optimization techniques to 
compute local maximizers of either $\ntfc(\omega)$ of $\ntfd(\theta)$
can itself be considered an efficient and scalable approach for approximating 
the $\Hinf$ norm of large-scale systems.
It is perhaps a bit staggering
that such a simple and direct approach seems to have been until now overlooked,
particularly given the sophistication of the existing $\Hinf$ norm approximation methods.

In more detail,
recall that the initialization phase of Algorithm~\ref{alg:hybrid_opt},
lines~\ref{algline:init_start}-\ref{algline:init_end}, 
is simply just applying unconstrained optimization to find one or more maximizers of
$\ntfc(\omega)$.
Provided that $(\imagunit \omega E - A)$ permits fast linear solves, 
e.g. a sparse LU decomposition,
there is no reason why this cannot also be done for large-scale systems.
In fact, the methods of \cite{BenV14,FreSV14,AliBMetal17} for approximating the $\Hinf$ norm 
all require that such fast solves are possible (while the methods of \cite{GugGO13,MitO16}
only require fast matrix-vector products with the system matrices).
When $m,p \ll n$, it is still efficient to calculate second derivatives of $\ntfc(\omega)$
to obtain a quadratic rate of convergence via Newton's method.
Even if $m,p \ll n$ does not hold, first derivatives of $\ntfc(\omega)$ can still be computed using sparse methods
for computing the largest singular value (and its singular vectors)
and thus the secant method can be employed to at least get superlinear convergence.
As such, the local convergence and superlinear/quadratic convergence rate guarantees of the existing methods
are at least matched by the guarantees of direct optimization.
For example, while the superlinearly-convergent method of \cite{AliBMetal17} requires that $m,p \ll n$,
our direct optimization approach remains efficient even if $m,p \approx n$, when it also has superlinear convergence,
and it has quadratic convergence in the more usual case of $m,p \ll n$.

Of course, there is also the question of whether there are differences in approximation quality between the methods.
This is a difficult question to address since beyond local optimality guarantees, there are no other theoretical results
concerning the quality of the computed solutions.  
Actual errors can only be measured when running the methods on small-scale systems, where the 
exact value of the $\Hinf$ norm can be computed,
while for large-scale problems, only relative differences between the methods' respective approximations
can be observed.  Furthermore, any of these observations may not be predictive of performance on other problems.  
For nonconvex optimization, 
the quality of a locally optimal computed solution is often dependent on the starting point,
which will be a strong consideration for the direct optimization approach.
On the other hand, it is certainly plausible that the sophistication of the existing $\Hinf$ norm algorithms
may favorably bias them to converge to better (higher) maximizers more frequently than 
direct optimization would, particularly if only randomly selected starting points were used.
With such complexities, in this paper we do not attempt to do a comprehensive benchmark with respect to existing $\Hinf$ 
norm approximation methods but only attempt to demonstrate that direct optimization is a potentially viable alternative.

We implemented lines~\ref{algline:init_start}-\ref{algline:init_end} of 
Algorithm~\ref{alg:hybrid_opt} in a second, standalone routine, 
with the necessary addition for the continuous-time case that 
the value of $\|D\|_2$ is returned if the computed local maximizers of $\ntfc(\omega)$ only yield lower function values than $\|D\|_2$. 
Since we assume that the choice of starting points will be critical, we initialized our sparse routine using starting frequencies 
computed by \texttt{samdp}, a \matlab\ code that implements the subspace-accelerated dominant pole algorithm of \cite{RomM06}.
Correspondingly, we compared our approach to the \matlab\ code \texttt{hinorm}, which implements
the spectral-value-set-based method using dominant poles of \cite{BenV14} and also uses \texttt{samdp} 
(to compute dominant poles at each iteration).
We tested \texttt{hinorm} using its default settings, and since it initially computes 20 dominant poles to find a good starting point,
we also chose to compute 20 dominant poles via \texttt{samdp} to obtain 20 initial frequency guesses for optimization.\footnote{Note 
that these are not necessarily the same 20 dominant poles, since \cite{BenV14} must first transform a system if the original system has nonzero $D$ matrix.}
Like our small-scale experiments, we also ensured zero was always included as an initial guess 
and reused the same choices for \texttt{fmincon} parameter values.
We tested our optimization approach by optimizing $\nopt=1,5,10$ of the most promising frequencies,
again using a serial MATLAB code.
Since we used LU decompositions to solve the linear systems, we tested our code in two configurations: with and without permutations, 
i.e. for some matrix given by variable \texttt{A}, \texttt{[L,U,p,q] = lu(A,'vector')} and \texttt{[L,U] = lu(A)}, respectively.

Table~\ref{table:problems_large} shows our selection or large-scale 
test problems, all continuous-time since \texttt{hinorm} does not support discrete-time problems (in contrast to our optimization-based approach which supports both).  Problems \texttt{dwave} and \texttt{markov} are from the large-scale test set used in \cite{GugGO13} while 
the remaining problems are freely available from the website of Joost Rommes\footnote{Available at \url{https://sites.google.com/site/rommes/software}}.  
As $m,p \ll n$ holds in all of these examples, we just present results
for our code when using Newton's method.
For all problems, our code produced $\Hinf$ norm approximations
that agreed to at least 12 digits with \texttt{hinorm}, meaning that the additional optimization calls done with $\nopt=5$ and 
$\nopt=10$ did not produce better maximizers than what was found with $\nopt=1$ and thus, only added to the serial computation running time.

\begin{table}[!t]
\centering
\begin{tabular}{ l | rrr | c  } 
\toprule
\multicolumn{5}{c}{Large-scale examples (continuous time)}\\
\midrule
\multicolumn{1}{l}{Problem} & 
	\multicolumn{1}{c}{$n$} & 
	\multicolumn{1}{c}{$p$} &
	\multicolumn{1}{c}{$m$} & 
	\multicolumn{1}{c}{$E=I$}  \\
\midrule
\texttt{dwave} & 2048 & 4 & 6 & Y \\
\texttt{markov} & 5050 & 4 & 6 & Y \\
\texttt{bips98\_1450} & 11305 & 4 & 4 & N \\
\texttt{bips07\_1693} & 13275 & 4 & 4 & N \\
\texttt{bips07\_1998} & 15066 & 4 & 4 & N \\
\texttt{bips07\_2476} & 16861 & 4 & 4 & N \\
\texttt{descriptor\_xingo6u} & 20738 & 1 & 6 & N \\
\texttt{mimo8x8\_system} & 13309 & 8 & 8 & N \\
\texttt{mimo28x28\_system} & 13251 & 28 & 28 & N \\
\texttt{ww\_vref\_6405} & 13251 & 1 & 1 & N \\
\texttt{xingo\_afonso\_itaipu} & 13250 & 1 & 1 & N \\
\bottomrule
\end{tabular}
\caption{The list of test problems for the large-scale $\Hinf$-norm approximation comparing
direct local optimization against \texttt{hinorm}, 
along with the corresponding problem dimensions and whether they are standard state-space systems ($E=I$) 
or descriptor systems $(E \ne I)$.}
\label{table:problems_large}
\end{table}

\begin{table}[!t]
\setlength{\tabcolsep}{3pt}
\robustify\bfseries
\center
\begin{tabular}{ l | SSS | SSS | S } 
\toprule
\multicolumn{8}{c}{Large-scale examples (continuous time)}\\
\midrule
\multicolumn{1}{c}{} & 
	\multicolumn{6}{c}{Direct optimization: $\nopt=1,5,10$ } & \multicolumn{1}{c}{\texttt{hinorm}} \\
\cmidrule(lr){2-7}
\cmidrule(lr){8-8}
\multicolumn{1}{c}{} & \multicolumn{3}{c}{\texttt{lu} with permutations} & \multicolumn{3}{c}{\texttt{lu} without permutations} & \multicolumn{1}{c}{}  \\
\cmidrule(lr){2-4}
\cmidrule(lr){5-7}
\multicolumn{1}{l}{Problem} & 
	\multicolumn{1}{c}{1} & \multicolumn{1}{c}{5} & \multicolumn{1}{c}{10} & 
	\multicolumn{1}{c}{1} & \multicolumn{1}{c}{5} & \multicolumn{1}{c}{10} & \multicolumn{1}{c}{} \\
\midrule	
\multicolumn{8}{c}{Wall-clock running times in seconds (initialized via \texttt{samdp})}\\
\midrule
\texttt{dwave}                   & 1.979 & 1.981 & 1.997 & 5.536 & 5.154 & 5.543 & \bfseries 1.861 \\ 
\texttt{markov}                  & \bfseries 3.499 & 3.615 & 3.593 & 26.734 & 26.898 & 27.219 & 3.703 \\ 
\texttt{bips98\_1450}            & \bfseries 6.914 & 8.333 & 10.005 & 14.559 & 17.157 & 19.876 & 31.087 \\ 
\texttt{bips07\_1693}            & \bfseries 8.051 & 9.155 & 11.594 & 18.351 & 21.367 & 24.322 & 75.413 \\ 
\texttt{bips07\_1998}            & \bfseries 10.344 & 11.669 & 13.881 & 50.059 & 56.097 & 59.972 & 51.497 \\ 
\texttt{bips07\_2476}            & \bfseries 14.944 & 16.717 & 18.942 & 65.227 & 70.920 & 73.206 & 76.697 \\ 
\texttt{descriptor\_xingo6u}     & 13.716 & 15.328 & 16.997 & \bfseries 7.907 & 9.225 & 11.133 & 36.775 \\ 
\texttt{mimo8x8\_system}         & 7.566 & 8.934 & 11.211 & \bfseries 6.162 & 7.562 & 9.321 & 30.110 \\ 
\texttt{mimo28x28\_system}       & 12.606 & 17.767 & 20.488 & \bfseries 10.815 & 16.591 & 21.645 & 33.107 \\ 
\texttt{ww\_vref\_6405}          & 7.353 & 6.785 & 7.552 & \bfseries 4.542 & 5.076 & 5.437 & 18.553 \\ 
\texttt{xingo\_afonso\_itaipu}   & 5.780 & 6.048 & 7.772 & \bfseries 4.676 & 4.975 & 5.573 & 16.928 \\
\midrule
\multicolumn{8}{c}{Running times relative to direct optimization (\texttt{lu} with permutations and $\nopt =1$)}\\
\midrule
\texttt{dwave}                   & 1 & 1.00 & 1.01 & 2.80 & 2.60 & 2.80 & \bfseries 0.94 \\ 
\texttt{markov}                  & \bfseries 1 & 1.03 & 1.03 & 7.64 & 7.69 & 7.78 & 1.06 \\ 
\texttt{bips98\_1450}            & \bfseries 1 & 1.21 & 1.45 & 2.11 & 2.48 & 2.87 & 4.50 \\ 
\texttt{bips07\_1693}            & \bfseries 1 & 1.14 & 1.44 & 2.28 & 2.65 & 3.02 & 9.37 \\ 
\texttt{bips07\_1998}            & \bfseries 1 & 1.13 & 1.34 & 4.84 & 5.42 & 5.80 & 4.98 \\ 
\texttt{bips07\_2476}            & \bfseries 1 & 1.12 & 1.27 & 4.36 & 4.75 & 4.90 & 5.13 \\ 
\texttt{descriptor\_xingo6u}     & 1 & 1.12 & 1.24 & \bfseries 0.58 & 0.67 & 0.81 & 2.68 \\ 
\texttt{mimo8x8\_system}         & 1 & 1.18 & 1.48 & \bfseries 0.81 & 1.00 & 1.23 & 3.98 \\ 
\texttt{mimo28x28\_system}       & 1 & 1.41 & 1.63 & \bfseries 0.86 & 1.32 & 1.72 & 2.63 \\ 
\texttt{ww\_vref\_6405}          & 1 & 0.92 & 1.03 & \bfseries 0.62 &  0.69 & 0.74 & 2.52 \\ 
\texttt{xingo\_afonso\_itaipu}   & 1 & 1.05 & 1.34 &  \bfseries 0.81 &0.86 & 0.96 & 2.93 \\ 
\midrule
Average                          & 1 & 1.12 & 1.30 & 2.52 & 2.74 & 2.97 & 3.70 \\  
\bottomrule
\end{tabular}
\caption{In the top half of the table, the running times (fastest in bold) are reported 
in seconds for our direct 
Newton-method-based optimization approach in two configurations 
(\texttt{lu} with and without permutations)
and \texttt{hinorm}.  Each configuration of our approach optimizes
the norm of the transfer function for up to $\nopt$ different starting frequencies ($\nopt=1,5,10$),
done sequentially.
The bottom half the table normalizes all the times relative to the running times for our 
optimization method using 
\texttt{lu} with permutations and $\nopt=1$,
along with the overall averages relative to this variant.
}
\label{table:times_large}
\end{table}

In Table~\ref{table:times_large}, we present the running times of the codes and configurations.  First, we observe that 
for our direct optimization code, using \texttt{lu} with permutations
is two to eight times faster than without permutations;
on average, using \texttt{lu} with permutations is typically 2.5 times faster.  
Interestingly, on the last five problems, using \texttt{lu} without permutations 
was actually best, but using permutations 
was typically only about 25\% slower and at worse, about 1.7 times slower 
(\texttt{descriptor\_xingo6u}).  
We found that our direct optimization approach, using just one starting
frequency ($\nopt=1$) was typically 3.7 times faster than \texttt{hinorm} on average and almost up to 10 times faster on problem
\texttt{bips07\_1693}.  Only on problem \texttt{dwave} was direct optimization actually slower than \texttt{hinorm} 
and only by a negligible amount.  Interestingly, optimizing just one initial frequency versus running optimization
for ten frequencies ($\nopt=10$) typically only increased the total running time of our code by 20-30\%.  This strongly suggested
that the dominant cost of running our code is actually just calling \texttt{samdp} to 
compute the 20 initial dominant poles to obtain starting guesses.  As such, in Table~\ref{table:times_samdp},
we report the percentage of the overall running time for each variant/method that was due to their initial calls to \texttt{samdp}. 
Indeed, our optimization code's single call to \texttt{samdp} accounted for 81.5-99.3\% of its
running-time 
(\texttt{lu} with permutations and $\nopt=1$).  
In contrast, \texttt{hinorm}'s initial call to \texttt{samdp} usually
accounted for about only a quarter of its running time on average, excluding \texttt{dwave} and \texttt{markov} as exceptional
cases.
In other words, the convergent phase of direct optimization is actually even faster than the convergent phase of 
\texttt{hinorm} than what Table~\ref{table:times_large} appears to indicate.  On problem
\texttt{bips07\_1693}, we see that our proposal to use Newton's method to optimize 
$\ntfc(\omega)$ directly is actually over 53 times faster than \texttt{hinorm}'s convergent phase.

\begin{table}[!t]
\center
\begin{tabular}{ l | SSS | SSS | S } 
\toprule
\multicolumn{8}{c}{Large-scale examples (continuous time)}\\
\midrule
\multicolumn{8}{c}{Percentage of time just to compute 20 initial dominant poles (first call to \texttt{samdp})}\\
\midrule
\multicolumn{1}{c}{} & 
	\multicolumn{6}{c}{Direct optimization: $\nopt=1,5,10$} & \multicolumn{1}{c}{\texttt{hinorm}} \\
\cmidrule(lr){2-7}
\cmidrule(lr){8-8}
\multicolumn{1}{c}{} & \multicolumn{3}{c}{\texttt{lu} with permutations} & \multicolumn{3}{c}{\texttt{lu} without permutations} & \multicolumn{1}{c}{}  \\
\cmidrule(lr){2-4}
\cmidrule(lr){5-7}
\multicolumn{1}{l}{Problem} & 
	\multicolumn{1}{c}{1} & \multicolumn{1}{c}{5} & \multicolumn{1}{c}{10} & 
	\multicolumn{1}{c}{1} & \multicolumn{1}{c}{5} & \multicolumn{1}{c}{10} & \multicolumn{1}{c}{} \\
\midrule	
\texttt{dwave}                   & 99.3 & 99.3 & 99.2 & 99.3 & 99.6 & 99.6 & 98.0 \\ 
\texttt{markov}                  & 99.2 & 99.1 & 99.1 & 99.6 & 99.6 & 99.6 & 98.1 \\ 
\texttt{bips98\_1450}            & 84.1 & 72.7 & 60.4 & 84.9 & 73.0 & 61.8 & 21.7 \\ 
\texttt{bips07\_1693}            & 84.2 & 75.1 & 61.6 & 85.6 & 77.1 & 64.4 & 10.1 \\ 
\texttt{bips07\_1998}            & 85.2 & 77.1 & 66.0 & 88.5 & 81.5 & 72.8 & 18.3 \\ 
\texttt{bips07\_2476}            & 88.0 & 80.2 & 71.4 & 89.2 & 83.6 & 73.0 & 18.6 \\ 
\texttt{descriptor\_xingo6u}     & 90.2 & 82.8 & 75.0 & 89.3 & 79.9 & 66.8 & 38.4 \\ 
\texttt{mimo8x8\_system}         & 84.6 & 74.4 & 59.1 & 84.1 & 67.7 & 55.0 & 24.2 \\ 
\texttt{mimo28x28\_system}       & 81.5 & 59.4 & 51.2 & 75.5 & 46.9 & 39.1 & 40.9 \\ 
\texttt{ww\_vref\_6405}          & 95.1 & 79.6 & 63.9 & 91.9 & 78.3 & 75.1 & 33.7 \\ 
\texttt{xingo\_afonso\_itaipu}   & 90.2 & 84.4 & 76.8 & 89.3 & 83.7 & 74.7 & 33.6 \\ 
\bottomrule
\end{tabular}
\caption{The column headings are the same as in Table~\ref{table:times_large}.}
\label{table:times_samdp}
\end{table}

\section{Parallelizing the algorithms}
\label{sec:parallel}
The original BBBS algorithm, and the cubic-interpolation refinement, only 
provide little opportunity for parallelization
at the \emph{algorithmic} level, i.e. when not considering that
the underlying basic linear algebra operations may be parallelized themselves when running on a single shared-memory multi-core machine.
Once the imaginary eigenvalues\footnote{
For conciseness, our discussion in Section~\ref{sec:parallel} 
will be with respect to the continuous-time case
but note that it applies equally to the discrete-time case as well.}
have been computed,
constructing the level-set intervals 
(line~\ref{algline:bbbs_ints} of Algorithm~\ref{alg:bbbs})
and calculating $\ntfc(\omega)$ at their midpoints or cubic-interpolant-derived maximizers 
(line~\ref{algline:bbbs_points} of Algorithm~\ref{alg:bbbs})
can both be done in an embarrassingly parallel manner, e.g. across nodes on a cluster.
However, as we have discussed to motivate our improved algorithm, 
evaluating $\ntfc(\omega)$ is a rather cheap operation compared to computing the 
eigenvalues of $\MNpencont$.
Crucially, parallelizing these two steps does not 
result in an improved (higher) value of $\gamma$ found per iteration
and so the number of expensive eigenvalue computations of 
$\MNpencont$ remains the same.  

For our new method, we certainly can (and should) also parallelize the construction of the level-set intervals 
and the evaluations of their midpoints or cubic-interpolants-derived maximizers
(lines~\ref{algline:hyopt_ints} and \ref{algline:hyopt_points} in Algorithm~\ref{alg:hybrid_opt}),
despite that we do not expect large gains to be had here.
However, optimizing over the intervals 
(line~\ref{algline:hyopt_opt} in Algorithm~\ref{alg:hybrid_opt}) is also an embarrassingly parallel task
and here significant speedups can be obtained.
As mentioned earlier, with serial computation (at the algorithmic level), we typically recommend only optimizing 
over a single level-set interval ($\nopt=1$) out of the $q$ candidates 
(the most promising one, 
as determined by line~\ref{algline:hyopt_reorder} in Algorithm~\ref{alg:hybrid_opt});
otherwise, the increased number of evaluations of $\ntfc(\omega)$ can start to outweigh the benefits
of performing the local optimization.
By optimizing over more intervals in parallel, e.g. again across nodes on a cluster,
we increase the chances on every iteration
of finding even higher peaks of $\ntfc(\omega)$, and possibly a global maximum, 
\emph{without any increased time penalty} (besides communication latency).\footnote{Note that 
distributing the starting points for optimization and taking the max of the resulting optimizers 
involves very little data being communicated on any iteration.}
In turn, larger steps in $\gamma$ can be taken, potentially reducing 
the number of expensive eigenvalue computations of $\MNpencont$ incurred.
Furthermore, parallelization can also be applied to the initialization stage 
to optimize from as many starting points as possible without time penalty 
(lines~\ref{algline:init_start} and \ref{algline:init_opt} in Algorithm~\ref{alg:hybrid_opt}),
a standard technique for nonconvex optimization problems.
Finding a global maximum of $\ntfc(\omega)$ during initialization means that the algorithm 
will only need to compute the eigenvalues of $\MNpencont$
just once, to assert that maximum found is indeed a global one.

When using direct local optimization techniques for $\Hinf$ approximation,
as discussed in Section~\ref{sec:parallel},
optimizing from as many starting points as possible of course
also increases the chances of finding the true value of the $\Hinf$ norm,
or at least better approximations than just starting from one point.
With parallelization, 
these additional starting points can also be tried without any time penalty
(also lines~\ref{algline:init_start} and \ref{algline:init_opt} in Algorithm~\ref{alg:hybrid_opt}),
unlike the experiments we reported in Section~\ref{sec:hinf_approx}
where we optimized using $\phi=1,5,10$ starting guesses with a serial MATLAB code
and therefore incurred longer running times as $\phi$ was increased.

For final remarks on parallelization, first note that there will generally be less benefits when
using more than $n$ parallel optimization calls, since there are at most $n$ peaks of $\ntfc(\omega)$.
However, for initialization, one could simply try as many starting guesses as there are parallel nodes available
(even if the number of nodes is greater than $n$)
to maximize the chances of finding a high peak of $\ntfc(\omega)$ or a global
maximizer.  Second, the number of level-set intervals encountered
by the algorithm at each iteration may be significantly less than $n$, particularly if good starting 
guesses are used.  Indeed, it is not entirely uncommon for the algorithm to only
encounter one or two level-set intervals on each iteration.  
On the other hand, for applications where $\ntfc(\omega)$ has many similarly high
peaks, such as controller design where the $\Hinf$ norm is minimized,
our new algorithm may consistently benefit from parallelization with a higher number of parallel optimization calls.

\section{Conclusion and outlook}
\label{sec:wrapup}
We have presented an improved algorithm that significantly reduces
the time necessary to compute the $\Hinf$ norm of linear control systems 
compared to existing algorithms.
Furthermore, our proposed hybrid optimization approach
also allows the $\Hinf$ norm to be computed to machine precision with relatively little extra work, 
unlike earlier methods.  We have also demonstrated that 
approximating the $\Hinf$ norm of large-scale problems via directly optimizing the norm of the transfer function is not only viable but can be quite efficient. 
In contrast to the standard BBBS and cubic-interpolation refinement algorithms,
our new approaches for $\Hinf$ norm computation and approximation also can
benefit significantly more from parallelization.   
Work is ongoing to add implementations of our new algorithms 
to a future release of the open-source library ROSTAPACK: RObust STAbility PACKage.\footnote{Available 
at \url{http://www.timmitchell.com/software/ROSTAPACK}}
This is being done in coordination with our efforts to also add implementations of our new 
methods for computing the spectral value set abscissa and radius, proposed in \cite{BenM17a} and which use related ideas to those in this paper.
The current v1.0 release of ROSTAPACK
contains implementations of scalable algorithms for approximating all of these aforementioned measures \cite{GugO11,GugGO13,MitO16}, 
as well as variants where the uncertainties are restricted to be real valued \cite{GugGMetal17}.

Regarding our experimental observations, the
sometimes excessively longer compute times for \texttt{hinfnorm} compared to
all other methods we evaluated possibly indicates that the structure-preserving eigensolver
that it uses can sometimes be much slower than QZ.  This certainly warrants further investigation,
and if confirmed, suggests that optimizing the code/algorithm of the structure-preserving eigensolver  
could be a worthwhile pursuit.  In the large-scale setting, we have observed that the dominant cost 
for our direct optimization approach is actually due to obtaining the starting frequency guesses via  
computing dominant poles.  If the process of obtaining good initial guesses can be accelerated, then 
approximating the $\Hinf$ norm via direct optimization could be significantly sped up even more.

\section{Acknowledgements}
The authors are extremely grateful to both referees for their many useful comments toward improving the paper.

\small
\bibliographystyle{alpha}
\bibliography{csc,mor,software}  

\end{document}